\documentclass[A4paper,leqno]{amsart}

\usepackage{hyperref}
\usepackage{aliascnt}
\usepackage{amsmath, amssymb, latexsym, graphicx,tikz, stmaryrd, ifsym, cleveref, enumerate, multicol} 

\usepackage{stackengine,scalerel}
\usepackage{esvect}
\usepackage[shortlabels]{enumitem}
\usepackage[latin1]{inputenc}
\usepackage{amssymb}
\usepackage{amsfonts}
\usepackage{amscd}
\usepackage{multicol}
\usepackage{mathrsfs}
\usepackage{enumerate}
\usepackage[T1]{fontenc}
\usepackage[english]{babel}
\usepackage{marvosym}
\usepackage{amsthm}
\usepackage{multicol}
\usepackage[all]{xy}
\usepackage{todonotes}
\usepackage{graphicx,tikz}
\usepackage{tikz-cd} 
\usepackage{mathbbol}
\usepackage{verbatim}
\usepackage{xcolor} 

\usepackage{stackengine}

\usepackage{pb-diagram}

\usepackage{bm}

\usepackage{wasysym}
\usepackage{color}

\theoremstyle{plain}
\newtheorem{thm}{Theorem}[section]
\newtheorem*{thm*}{Theorem}

\newaliascnt{lem}{thm}
\newtheorem{lem}[lem]{Lemma}
\aliascntresetthe{lem}

\newaliascnt{claim}{thm}

\aliascntresetthe{claim}

\newaliascnt{prop}{thm}
\newtheorem{prop}[prop]{Proposition}
\aliascntresetthe{prop}
\newtheorem*{prop*}{Proposition}

\newaliascnt{conj}{thm}

\aliascntresetthe{conj}

\newaliascnt{coro}{thm}
\newtheorem{coro}[coro]{Corollary}
\aliascntresetthe{coro}

\newaliascnt{belief}{thm}

\aliascntresetthe{belief}

\newtheorem{thmIntro}{Theorem}

\theoremstyle{remark}

\newaliascnt{term}{thm}

\aliascntresetthe{term}

\newaliascnt{rk}{thm}
\newtheorem{rk}[rk]{Remark}
\aliascntresetthe{rk}

\newaliascnt{df}{thm}
\newtheorem{df}[df]{Definition}
\aliascntresetthe{df}

\newaliascnt{ex}{thm}
\newtheorem{ex}[ex]{Example}
\aliascntresetthe{ex}

\newaliascnt{setup}{thm}

\aliascntresetthe{setup}

\newaliascnt{nota}{thm}

\aliascntresetthe{nota}

\crefname{thm}{theorem}{theorems}
\Crefname{thm}{Theorem}{Theorems}

\crefname{lem}{lemma}{lemmas}
\Crefname{lem}{Lemma}{Lemmas}

\crefname{claim}{claim}{claims}
\Crefname{claim}{Claim}{Claims}

\crefname{prop}{proposition}{propositions}
\Crefname{prop}{Proposition}{Propositions}

\crefname{conj}{conjecture}{conjectures}
\Crefname{conj}{Conjecture}{Conjectures}

\crefname{coro}{corollary}{corollaries}
\Crefname{coro}{Corollary}{Corollaries}

\crefname{belief}{belief}{beliefs}
\Crefname{belief}{Belief}{Beliefs}

\crefname{term}{terminology}{terminologies}
\Crefname{term}{Terminology}{Terminologies}

\crefname{rk}{remark}{remarks}
\Crefname{rk}{Remark}{Remarks}

\crefname{df}{definition}{definitions}
\Crefname{df}{Definition}{Definitions}

\crefname{ex}{example}{examples}
\Crefname{ex}{Example}{Examples}

\crefname{setup}{setup}{setups}
\Crefname{setup}{Setup}{Setups}

\crefname{nota}{notation}{notations}
\Crefname{nota}{Notation}{Notations}

\numberwithin{equation}{section}

\def\1c{\mathrm{1c}}

\def\bbC{\mathbb{C}}

\def\bbF{\mathbb{F}}

\def\bbZ{\mathbb{Z}}

\def\calC{\mathcal{C}}

\def\calP{\mathcal{P}}

\newcommand\restr[2]{{
  \left.\kern-\nulldelimiterspace 
  #1 
  \vphantom{\big|} 
  \right|_{#2} 
  }}

\def\homo{\operatorname{\it \mathscr{H}\kern-.25em om}}
\def\ext{\operatorname{\it \mathscr{E}\kern-.25em xt}}
\def\edo{\operatorname{\it \mathscr{E}\kern-.25em nd}}
\def\der{\operatorname{\it \mathscr{D}\kern-.25em er}}

\def\Hom{\mathrm{Hom}}

\def\End{\mathrm{End}}

\def\Ker{\mathrm{Ker}}
\def\Im{\mathrm{Im}}

\newcommand{\Zd}{\bbZ_d}

\usetikzlibrary{decorations.pathreplacing,backgrounds,decorations.markings,arrows}
\tikzset{wei/.style={draw=red,double=red!40!white,double distance=1.5pt,thin}}
\tikzset{bdot/.style={fill,circle,color=blue,inner sep=3pt,outer sep=0}}
\tikzset{dir/.style={postaction={decorate,decoration={markings,mark=at position .8 with {\arrow[scale=1.3]{>}}}}}}

\newcommand{\C}{\mathbb{C}}

\newcommand{\Aut}{{\rm Aut}}

\makeatletter
\newcommand{\mywedge}{\@ifnextchar^\@extp{\@extp^{\,}}}
\def\@extp^#1{\mathop{\bigwedge\nolimits^{\!\!#1}}}
\makeatother

\newcommand{\sM}{{F_n}} 

\newcommand{\dbar}{{\bar d}}

\begin{document}

\title[Qudit stabilizers and the twisted Kitaev model]{Qudit stabilizers beyond the free case and the twisted Kitaev model}
\author[R. Maksimau]{Ruslan Maksimau}

\begin{abstract}	
We study the stabiliser formalism for qudits of arbitrary dimension $d$. In the free case, we show that the basic theorem of the stabiliser formalism remains valid: if the stabiliser subgroup $H$ is free as a $\bbZ/d\bbZ$-module and contains no non-trivial scalars, then the protected space $V^H$ is naturally identified with the state space of a smaller number of qudits of the same dimension, and the quotient $N(H)/H$ is identified with the Pauli group on a smaller number of qudits.

We then remove the freeness assumption and describe the resulting structure in general. In this case, the protected space is identified with a tensor product of qudit spaces of possibly smaller dimensions, and the quotient $N(H)/H$ is described by a corresponding product of qudit Pauli groups, possibly of smaller dimensions, over a common center. We also characterise the shifted free case, which is exactly the situation in which $N(H)/H$ is again an ordinary qudit Pauli group. Our approach is algebraic and uniform, and applies in particular to the qudit Kitaev model and to its shifted and twisted variants.

\end{abstract}

\maketitle
\vspace{-5mm}

\tableofcontents

\section{Introduction}

In both classical and quantum computation, one has to deal with errors affecting the information being processed. For this, we may need to perform error correction procedures. One of the simplest classical examples is the repetition code, where each bit $b\in\{0,1\}$ is encoded as three identical bits $\overline{bbb}\in \{0,1\}^3$. In case of an error affecting one of the bits, the resulting triple is no longer constant, and this allows us to detect and correct the error. Such an error correction code may be described by a map from $\{0,1\}^n$ to $\{0,1\}^{3n}$, meaning that we encode $n$ logical bits using $3n$ physical bits. More generally, we may want to encode $k$ logical bits into $n$ physical bits, which corresponds to a map $\{0,1\}^k\hookrightarrow\{0,1\}^{n}$, together with a procedure for detecting and correcting errors.

For quantum computation we use qubits instead of bits. A state of a qubit is not an element of $\{0,1\}$, but an element of a 2-dimensional hermitian vector space with an orthonormal basis $|0\rangle$, $|1\rangle$, we will call this space just $\bbC^2$.
For readers not familiar with the basics of quantum mechanics, we refer to the author's notes for a talk at a working group \cite{maksimau2025spin}, which provide an elementary introduction to qubits and related notions in a language aimed at algebraists.

 Now, if we want to encode $k$ logical qubits with $n$ physical qubits, this corresponds to finding a linear map $(\bbC^2)^{\otimes k}\hookrightarrow (\bbC^2)^{\otimes n}$. Gottesman \cite{gottesman1997stabilizer} proposed an approach for constructing such linear maps. Let us recall briefly Gottesman's stabiliser formalism.

Let $\calP_n$ be the $n$-qubit Pauli group acting on the $n$-qubit vector space $V=(\bbC^2)^{\otimes n}$. For a subgroup $H\subset \calP_n$, set $V^H=\{v\in V|\, \forall h\in H,\, h(v)=v\}$. Denote by $N(H)$ the normaliser of $H$ in $\calP_n$.
\begin{prop*}[Gottesman]
	Let $H\subset \calP_n$ be an arbitrary abelian subgroup not containing nontrivial scalars and let $2^k$ be the cardinality of $H$. Then there is a unitary isomorphism $V^H\simeq (\bbC^2)^{\otimes n-k}$ and an isomorphism of groups $N(H)/H\simeq \calP_{n-k}$ such that the action of $N(H)/H$ on $V^H$ is identified with the action of $\calP_{n-k}$ on $(\bbC^2)^{\otimes n-k}$.
\end{prop*}

The proposition above allows one to construct interesting linear maps of the form $(\bbC^2)^{\otimes n-k}\hookrightarrow (\bbC^2)^{\otimes n}$ and to obtain all interesting operators on $(\bbC^2)^{\otimes n-k}$ as restrictions of operators on $V=(\bbC^2)^{\otimes n}$ preserving $V^H$.

Now, we would like to generalise this result when we move from qubits ($\bbC^2$) to qudits ($\bbC^d$) for an arbitrary positive integer $d$. The aim of this paper is to give a uniform algebraic formulation of the stabiliser formalism for qudits. The emphasis is not on constructing new codes, but on providing a conceptual framework in which the structure of stabiliser codes becomes transparent.

We work with the $n$-qudit Pauli group $\calP_n=\calP_n^{(d)}$ acting on $V=(\bbC^d)^{\otimes n}$. The proof of the proposition above for qubits is based on some symplectic linear algebra over $\bbF_2=\bbZ/2\bbZ$ applied to the $\bbF_2$-vector space obtained when we quotient the group $\calP_n$ by its center (this vector space has a natural symplectic form). It is easy to generalise Gottesman's result to the situation when $d=p$ is prime because we can then work over the finite field $\bbF_p$. However, for a general $d$, the situation is not so easy. We cannot use vector spaces over a field, but we have to work with modules over the ring $\bbZ_d=\bbZ/d\bbZ$. This ring is far from being a field. Moreover it has zero divisors. This leads to additional technical difficulties.

The dimension formula for stabiliser spaces $V^H$ in the qudit setting is known, see for instance \cite[Theorem 1]{gheorghiu2014standard}. The approach developed in our paper goes beyond dimension counting, as it also describes the induced action of $N(H)/H$ on $V^H$ and identifies it with a Pauli-type group in a uniform algebraic framework.

Among the $\Zd$-modules, the simplest ones are free $\Zd$-modules, they are the closest to vector spaces over a field. (For $d=p$ prime, all modules are free.) We generalise Gottesman's result in two steps. First, we assume that $H$ is a free $\Zd$-module. At the second step, we remove the "free" assumption and study how the structure of $V^H$ and the action of $N(H)/H$ on it look in general.

\Cref{thmA} below is the main result of the paper in the "free" case. It shows that the obvious natural generalisation of Gottesman's result is true when we assume that $H$ is free.

\begin{thmIntro}[\Cref{thm:main-free-dstab}]
	\label{thmA}
	Assume that $H\subset \calP_n^{(d)}$ is a commutative subgroup without non-trivial scalars such that $H$ is a free $\Zd$-module of rank $k$. Then there is a unitary isomorphism $V^H\simeq (\bbC^d)^{\otimes n-k}$ and an isomorphism of groups $N(H)/H\simeq \calP^{(d)}_{n-k}$ such that the action of $N(H)/H$ on $V^H$ is identified with the action of $\calP^{(d)}_{n-k}$ on $(\bbC^d)^{\otimes n-k}$.
\end{thmIntro}

Next, in \Cref{sec:Kitaev}, we show how this result can be applied to the study of the Kitaev model \cite{kitaev2003fault} for qudits. This section contains no new results. However, it is important to note that, as in the case $d=2$, the Kitaev model can now be treated within the stabiliser formalism in a uniform way, so that its main properties follow directly from the general framework, without the need for ad hoc arguments.

Next, we would like to understand what happens when we remove the assumption "free". Set $\dbar=d$ if $d$ is odd and $\dbar=2d$ if $d$ is even. Let $\zeta\in \bbC$ be a primitive $\dbar$-th root of unity. For a positive divisor $d_0$ of $d$, let $\calP_n^{(d_0,\zeta)}$ be the subgroup of $(\bbC^{d_0})^{\otimes n}$ generated by $\calP_n^{(d_0)}$ and by the operator of multiplication by $\zeta$.

The following theorem is the main "non-free" result of the paper.
\begin{thmIntro}[\Cref{thm:main-gen+rep}]	
	Assume that $H\subset \calP_n^{(d)}$ is a commutative subgroup without non-trivial scalars. Then we can find positive divisors $d_1,d_2,\ldots,d_k$ of $d$ such that the following holds. There is a unitary isomorphism 
	$$
	V^H\simeq \bbC^{d_1}\otimes \bbC^{d_2}\otimes \ldots \otimes\bbC^{d_k}
	$$
	and an isomorphism of groups 
	$$
	N(H)/H\simeq \calP_1^{(d_1,\zeta)}\times_\zeta \calP_1^{(d_2,\zeta)}\times_\zeta \ldots\times_\zeta\calP_1^{(d_k,\zeta)}
	$$ 
	such that the action of $N(H)/H$ on $V^H$ is identified with the action of 
	$$
	\calP_1^{(d_1,\zeta)}\times_\zeta \calP_1^{(d_2,\zeta)}\times_\zeta \ldots\times_\zeta\calP_1^{(d_k,\zeta)} \quad\mbox{on}\quad \bbC^{d_1}\otimes \bbC^{d_2}\otimes \ldots \otimes\bbC^{d_k}.
	$$
\end{thmIntro}	
Above, the notation $G_1\times_\zeta G_2$ means the direct product of groups with the operators of multiplication by $\zeta$ identified in both groups.

Finally, we would like to identify the minimal generalisation of the assumption "free" such that the conclusion of \Cref{thmA} still holds. We introduce a procedure of shifting a free group $H$. We get the following theorem.
\begin{thmIntro}[\Cref{thm:shifted-free}]
	Let $H$ be an abelian subgroup of $\calP^{(d)}_n$ without non-trivial scalars. Then the group $N(H)/H$ is isomorphic to $\calP^{(d)}_{n-k}$ for some $k$ if and only if $H$ is shifted free. Moreover, in this case the action of $N(H)/H$ on $V^H$ can be identified with the action of $\calP^{(d)}_{n-k}$ on $(\bbC^d)^{\otimes (n-k)}$, where the isomorphism $V^H\simeq (\bbC^d)^{\otimes (n-k)}$ is unitary.
\end{thmIntro}

We also define the notion of a twist, which is a slight generalisation of the notion of a shift, and we study how it changes the structure of $N(H)/H$ and $V^H$. Finally, we apply this notion to the Kitaev model. We introduce the notions of the shifted and the twisted Kitaev model. The study of such modifications of the Kitaev model becomes automatic in view of our generalised stabiliser formalism for qudits. This illustrates the advantage of the approach and the philosophy of using the stabiliser formalism as a general tool instead of finding special arguments to study a given model. 
We would like to stress that shifts and twists are not modifications introduced specifically for the Kitaev model, but rather natural constructions arising from the general stabiliser formalism.

The structure of the paper is as follows. For pedagogical reasons we split it into two parts: Part I (the free case) and Part II (the general case). The theory in the non-free case automatically covers the free case. However, we decided to make this separation because different parts may interest different audiences.

Before these two parts there is a preliminary Part 0 studying symplectic $\bbZ/d\bbZ$-modules. It contains no new results. However, it prepares all lemmas about $\bbZ/d\bbZ$-modules that we need in the required generality. Next, Part I is split into \Cref{sec:free-stab}, describing the general theorem of the stabiliser formalism in the free case, and \Cref{sec:Kitaev}, which shows how it can be applied to examples such as CSS codes and the Kitaev toric code. Next, Part II about the general (non-free) theory is also split into two sections. In \Cref{sec:not-free}, we give the general results about the non-free stabiliser formalism. In \Cref{sec:Kitaev-tw-sh}, we apply them to the shifted and twisted Kitaev model.

This paper is written with two different audiences in mind: on the one hand, researchers in quantum information, in particular those working on quantum error correction and stabiliser formalism; on the other hand, algebraists interested in the underlying algebraic structures, who may not be very familiar with quantum information.
As a consequence, some parts of the exposition may recall basic facts that are standard in one of these areas, while other parts may appear more condensed to readers from the other community. The author has attempted to keep a balance between these two perspectives and to present the material in a way that remains accessible to both audiences. In several proofs, it would be natural from an algebraic point of view to use homological algebra; however, we deliberately avoid it and give more direct arguments instead.

Even though this was not an original goal of the paper, the approach turns out to be sufficiently self-contained so that all statements from quantum information that appear in the paper are proved within our framework, without using well-known results from the literature as black boxes. We nevertheless refer to the original works to give proper credit to their authors.

In several cases, the arguments presented here provide shorter and more algebraic proofs of known results, whose original proofs are often more algorithmic or computational in nature. This reflects the fact that, once the algebraic framework is set up, these results arise naturally within it and can be obtained with relatively little additional effort.

For instance, we recover the following known results:
\begin{itemize}
	\item the surjectivity of the map from the Clifford group to the symplectic group \cite{hostens2005stabilizer},
	\item the classification of Lagrangian submodules of a free symplectic $\bbZ/d\bbZ$-module \cite{albouy2008matrix},
	\item the dimension of the protected space $V^H$ \cite{gheorghiu2014standard}.
\end{itemize}

In this sense, the paper may also serve as a concise introduction for algebraists to some algebraic aspects of the stabiliser formalism and quantum error correction, presented in a purely algebraic language.

\bigskip
\newpage
\addtocontents{toc}{{\textbf{Part 0: Preparation}}}
\centerline{\LARGE Part 0: Preparation}
\section{Symplectic $(\bbZ/d\bbZ)$-modules}
\label{sec:sympl-mod}

\subsection{Notation and conventions}
Let us abbreviate $\Zd$ for $\bbZ/d\bbZ$. It will however sometimes be useful to use the notation $\bbZ/d\bbZ$ when we want to stress that it is the quotient of the ring $\bbZ$ by the ideal $d\bbZ$.

We start with a remark about the notion of a $\Zd$-module. Usually, when $R$ is an arbitrary ring, when we say "$M$ is an $R$-module", we mean that $M$ is an abelian group with an additional structure of multiplication by elements of $R$, satisfying some natural axioms.

However, for $R=\Zd$, it is more reasonable to think that being a $\Zd$-module is not an additional structure, but an additional property of an abelian group $M$. Indeed, since each abelian group $M$ has a unique structure of a $\bbZ$-module, there is no difference between being a $\bbZ$-module and being an abelian group. Now, for a $\bbZ$-module $M$, the action of $\bbZ$ factors through the action of $\Zd=\bbZ/d\bbZ$ if and only if the ideal $d\bbZ$ acts by zero on $M$, i.e. if and only if the order of each element of $M$ divides $d$. So, finally, this means that there is no difference between "$M$ is a $\Zd$-module" and "$M$ is an abelian group such that the order of each element divides $d$". Moreover, for a $\Zd$-module $M$, there is no difference between the notions "a $\Zd$-submodule of $M$" and "a subgroup of $M$".

For $a,b\in \bbZ$ such that $a\leqslant b$ we set $[a;b]=\{a,a+1,\ldots,b-1,b\}$. We mean that for $a=b$ this set contains one element. For a finite set $X$ denote by $\#X$ the number of elements in $X$.

For some elements $e_1,e_2,\ldots,e_n$ of a $\Zd$-module $M$ we denote by $\langle e_1,e_2,\ldots,e_n\rangle$ the submodule generated by them (i.e. all $\Zd$-linear combinations). Recall that we say that $(e_1,\ldots,e_n)$ is a basis of $M$ if 
$M=\langle e_1,\ldots,e_n\rangle$ and if any relation 
$\sum_{i=1}^n \lambda_i e_i = 0$ with $\lambda_i\in \Zd$ 
implies $\lambda_1=\cdots=\lambda_n=0$. A $\Zd$-module is called \emph{free} if it admits a basis. (We only work with finite $\Zd$-modules in this paper.)

Abusing the notation, we often use the same symbol for an integer $a\in\bbZ$ and its image in $\Zd$. When we write fractions, we always mean that we divide in $\bbZ$ (this is only possible when the division is exact), and sometimes we then consider the result as an element of $\Zd$. For example, for a divisor $a$ of $d$, we may write $\frac{d}{a}\in\Zd$, meaning that we first divide $d$ by $a$ in $\bbZ$ and then consider the image of the result in $\Zd$.

\subsection{Free $\Zd$-modules}

The following proposition recalls the well-known classification of finite $\Zd$-modules.
\begin{prop}
	\label{prop:classif-Zd-mod}
	Let $N$ be a finite $\Zd$-module. Then there exists a unique tuple of positive integers $(d_1,d_2,\ldots,d_n)$ such that $1<d_1\mid d_2\mid\ldots\mid d_n\mid d$, and such that $N\simeq \bigoplus_{r=1}^n\bbZ_{d_r}$.
\end{prop}

The following proposition, known as the Smith normal form, describes the submodules of a free $\Zd$-module of finite rank.
\begin{prop}
	\label{prop-Smith}
	Assume that $M$ is a free $\Zd$-module of rank $n$ and $N$ is a submodule of $M$. Then there exist positive integers $a_1,a_2,\ldots,a_n$ and a $\Zd$-basis $e_1,e_2,\ldots,e_n$ of $M$ such that $a_1\mid a_2 \mid a_3 \mid\ldots \mid a_n \mid d$ and such that $N$ is generated by $a_1e_1$, $a_2e_2$, ..., $a_ne_n$.
\end{prop} 

Note that in the proposition above some $a_k$ is allowed to be $d$. In this case we get $a_ke_k=0$.

\begin{coro}
	\label{coro:free-Smith}
	If a submodule $N$ of a finite free $\Zd$-module $M$ is free, then $N$ is a direct factor of $M$ (i.e., we have $M=N\oplus N'$ for some submodule $N'\subset M$). Moreover, each such complement $N'$ is also free.
\end{coro}
\begin{proof}
	Consider the Smith normal form of $N$. Then the submodule $N$ is free over $\Zd$ if and only if there exists an index $k\in [0;n]$ such that $a_i=1$ if $1\leqslant i \leqslant k$ and $a_i=d$ if $k<i\leqslant n$. In this case it is obvious that $N$ is a direct factor. Moreover, each complement is isomorphic to the quotient module $M/N$, which is obviously free.
\end{proof}

\begin{coro}
	\label{coro:free-complete-basis}
	If $N$ is a free submodule of a finite free $\Zd$-module $M$, then each basis of $N$ can be completed to a basis of $M$.
\end{coro}
\begin{proof}
	This follows immediately from the corollary above.
\end{proof}

\subsection{Alternating form}

\begin{df}
	\label{def:alt-form}
	Let $M$ be a $\Zd$-module. Let $\phi\colon M\times M\to \Zd$ be a map. We say that $\phi$ is an \emph{alternating form} if it is $\Zd$-bilinear and satisfies $\phi(m,m)=0$ for each $m\in M$. We say that the $\Zd$-submodule $Ker\phi=\{m_0\in M |\, \forall m\in M,\phi(m_0,m)=0\}$ is the kernel of the form.
\end{df}
\begin{rk}
	Note that an alternating form automatically satisfies the property $\phi(m_1,m_2)=-\phi(m_2,m_1)$ for $m_1,m_2\in M$.
\end{rk}

Let $M$ be a $\Zd$-module. Denote by $M^*$ the dual $\Zd$-module $\Hom_{\Zd}(M,\Zd)$. 

\begin{lem}
	\label{lem:star-surj}
	Assume that $M$ is a finite $\Zd$-module and that $N$ is a submodule. Then the natural map $M^*\to N^*$ is surjective. 
\end{lem}
\begin{proof}
	Assume first that $M$ is a free $\Zd$-module. Then the statement follows immediately from the Smith normal form (\Cref{prop-Smith}).
	
	In general, we can embed $M$ into a finite free $\Zd$-module $\overline M$. Then the inclusions $N\subset M\subset \overline M$ yield 
	$
	\overline M^*\to M^*\to N^*.
	$
	Since $\overline M$ is free, the map $\overline M^*\to N^*$ is surjective, and therefore the map $M^*\to N^*$ is also surjective.
\end{proof}

\begin{lem}
	Let $M$ be a finite $\Zd$-module. Let $\phi\colon M\times M\to \Zd$ be an alternating form. Then the following statements are equivalent.
	
	$(a)$ We have $\Ker\phi=\{0\}$.
	
	$(b)$ The map $M\to M^*$, $m\mapsto \phi(m,\bullet)$ is bijective.
\end{lem}
\begin{proof}
	It is clear that condition $(a)$ is equivalent to the injectivity of the map in $(b)$. However, since $M$ and $M^*$ are finite sets of the same cardinality, there is no difference between the injectivity and the bijectivity of this map.
\end{proof}

From now on, we always assume that the $\Zd$-module $M$ is finite.
\begin{df}
	For an alternating form $\phi$ on a finite $\Zd$-module we say that it is \emph{symplectic} if it satisfies the equivalent conditions $(a)$ and $(b)$ above. We will use this term only for finite modules (where the conditions are equivalent).
	We may simply say "$(M,\phi)$ is symplectic", meaning that "$M$ is a finite $\Zd$-module with a symplectic form $\phi$". 
	
	Fix a symplectic form $\phi$ on $M$. For a submodule $N\subset M$ we set $N^\perp=\{m\in M|\,\forall n\in N, \phi(m,n)=0\}$.
\end{df}

\begin{lem}
	\label{lem:card-Nperp}
	We have $\#N^\perp = \frac{\#M}{\#N}$.
\end{lem}
\begin{proof}
	By \Cref{lem:star-surj}, the morphism $M^*\to N^*$ is surjective. If we identify $M$ with $M^*$ via the symplectic form $\phi$, then the kernel of this morphism is $N^\perp$. This implies $\#N^\perp=\frac{\#M^*}{\#N^*}=\frac{\#M}{\#N}$.
\end{proof}

\begin{coro}
	\label{coro:N-perp-perp}
	We have $(N^\perp)^\perp=N$.
\end{coro}
\begin{proof}
	We obviously have $N\subset (N^\perp)^\perp$. However, this inclusion must be an equality for cardinality reasons.
\end{proof}

\begin{df}
	Let $d_0$ be a positive integer divisor of $d$. The \emph{elementary symplectic $\Zd$-module} $S_{d_0}$ is the $\Zd$-module $\bbZ_{d_0}\oplus \bbZ_{d_0}$ with the symplectic form $\phi\colon (\bbZ_{d_0}\oplus\bbZ_{d_0})\times (\bbZ_{d_0}\oplus\bbZ_{d_0})\to \Zd$ given by $\phi((1,0),(0,1))=d/d_0$.
\end{df}

If $(M_1,\phi_1)$ and $(M_2,\phi_2)$ are two $\Zd$-modules with alternating forms, then we can define their (orthogonal) direct sum $(M_1\oplus M_2,\phi_1\oplus\phi_2)$ as the $\Zd$-module $M_1\oplus M_2$ equipped with the alternating form $\phi=\phi_1\oplus \phi_2$ given by 
$\phi((m_1,m_2),(m'_1,m'_2))=\phi_1(m_1,m'_1)+\phi_2(m_2,m'_2)$.
When we use the notation $S_{d_0}$, we understand that the symplectic form on it is fixed. Then we may speak about direct sums of symplectic modules of the form $S_{d_0}$ (and we understand that we consider the symplectic form obtained as the direct sum).

\begin{lem}
	\label{lem:structure-sympl}
	For each nonzero symplectic $(M,\phi)$, we can find positive integers $d_1,d_2,\ldots,d_k$ such that $1<d_1\mid d_2\mid d_3\mid\ldots \mid d_k\mid d$ and we have an isomorphism $M\simeq S_{d_1}\oplus S_{d_2}\oplus \ldots \oplus S_{d_k}$ of $\Zd$-modules respecting the symplectic form.
\end{lem}

\begin{proof}
	Let $m$ be an element of $M$ of maximal order and let $a$ be its order.
	
	By \Cref{lem:star-surj}, we can extend each $\Zd$-linear form from $\langle m \rangle$ to $M$, and then (using $M\simeq M^*$) we can find $m'\in M$ such that $\phi(m',m)=\frac{d}{a}$. Since the order $a$ of $m$ is maximal, the order of $m'$ is also $a$ (it cannot be smaller than $a$ because of $\phi(m',m)=\frac{d}{a}$).
	
	Let $N=\langle m,m'\rangle$ be the $\Zd$-submodule of $M$ generated by $m$ and $m'$. It is clear from the construction of $N$ that we have $N^\perp \cap N=\{0\}$. However, since $\#N\cdot \#N^\perp =\#M$ by \Cref{lem:card-Nperp}, we get $M=N\oplus N^\perp$. 
	
	We have $N\simeq S_a$. Since $a$ is the maximal order of an element of $M$, all other orders must be divisors of $a$. In particular, the orders of all elements of $N^\perp$ are divisors of $a$. Then we just apply induction to $N^\perp$, and we get that $N^\perp$ is isomorphic to a direct sum of $S_{a'}$ such that $a'$ are divisors of $a$ (and these $a'$ may be ordered linearly in a way that each $a'$ divides the next one).  
\end{proof}

\begin{coro}[Lagrangian decomposition]
	Assume that $(M,\phi)$ is symplectic. Then we can find a $\Zd$-module $L$ such that we have an isomorphism of $\Zd$-modules $M\simeq L\oplus L^*$ and under this isomorphism the form $\phi$ looks in the following way:
	$$
	((l_1,l'_1),(l_2,l_2'))\mapsto l'_1(l_2)-l'_2(l_1).
	$$
\end{coro}

\begin{proof}
	The statement is obvious for an elementary symplectic module $S_{d_0}$. Then it is automatically true for a direct sum of such modules. Then, by \Cref{lem:structure-sympl}, this holds for all symplectic modules.
\end{proof}

\subsection{Free symplectic modules}
\begin{lem}
	\label{lem:free-sympl}
	Assume that $M$ is finite and $\phi$ is alternating. Then the following two conditions are equivalent.
	\begin{itemize}
		\item[$(a)$] $M$ is free and $\phi$ is symplectic.
		\item[$(b)$] There is a basis $(e_1,\ldots,e_n,f_1,\ldots,f_n)$ of $M$ such that for all $i, j\in [1;n]$ we have $\phi(e_i,e_j)=\phi(f_i,f_j)=0$, $\phi(e_i,f_j)=\delta_{i,j}$ where $\delta_{i,j}$ is the Kronecker delta.
	\end{itemize}
\end{lem}
\begin{proof}
	$(b)\implies (a)$ is obvious. Let us show $(a)\implies (b)$. Assume that $M$ is free and $\phi$ is symplectic. Then we have $M\simeq S_{d_1}\oplus S_{d_2}\oplus \ldots \oplus S_{d_n}$ as in \Cref{lem:structure-sympl}. However, the module $M$ may be free only in the case $d_1=d_2=\ldots =d_n=d$ (by the uniqueness of $d_1,d_2,\ldots,d_n$ in \Cref{prop:classif-Zd-mod}). In this case we get a basis of the expected form.
\end{proof}

\begin{df}
	We say that $(e_1,\ldots,e_n,f_1,\ldots,f_n)$ is a \emph{symplectic basis} of a free symplectic $\Zd$-module $M$ if it is a basis satisfying the conditions in \Cref{lem:free-sympl} $(b)$.
\end{df}

From the point of view of the terminology introduced above, for a finite $M$ with an alternating form $\phi$, being free and symplectic is the same thing as the existence of a symplectic basis.
From now on we assume that $M$ is a finite free $\Zd$-module of rank $2n$ with a symplectic form $\phi$.

\begin{df}
	\label{df:isotr}
	Let $N$ be a $\Zd$-submodule of $M$.
	
	$(a)$ We say that $N$ is isotropic if $N\subset N^\perp$ (i.e., in other words, if the restriction of $\phi$ to $N$ is zero).
	
	$(b)$ We say that $N$ is Lagrangian if $N=N^\perp$ (i.e., in other words, if $N$ is isotropic and cannot be included into a strictly bigger isotropic submodule; in view of \Cref{lem:card-Nperp} this is also equivalent to "$N$ is isotropic of cardinality $\sqrt{\#M}$").
	
	$(c)$ We say that $N$ is a symplectic submodule if $N\cap N^\perp=\{0\}$ (i.e., in other words, the restriction of $\phi$ to $N$ is still symplectic).
\end{df}

\begin{rk}
	\label{rk:orth-sympl}
	Note additionally that, in view of \Cref{lem:card-Nperp}, the condition $N\cap N^\perp=\{0\}$ is equivalent to $M=N+N^{\perp}$ and then it is also equivalent to $M=N\oplus N^\perp$. In particular, every symplectic submodule has an orthogonal complement, which is also a symplectic submodule.
\end{rk}

\begin{lem}
	\label{lem:free-isotr}
Let $M$ be a finite free $\Zd$-module with a symplectic form $\phi$. Assume that $L$ is a free isotropic submodule of $M$. Then each basis $(e_1,\ldots,e_k)$ of $L$ may be included in a symplectic basis $(e_1,\ldots,e_n,f_1,\ldots,f_n)$ of $M$.
\end{lem}

\begin{proof}
	By \Cref{lem:star-surj}, the map $M^*\to L^*$ is surjective. Then the isomorphism $M\simeq M^*$ allows us to lift $\Zd$-linear forms in $L^*$ to elements of $M$. This allows us to construct elements $f_1,\ldots,f_k\in M$ such that $\phi(e_i,f_j)=\delta_{i,j}$. The elements $(e_1,\ldots,e_k,f_1,\ldots,f_k)$ are clearly $\Zd$-linearly independent (to see this, apply $\phi(e_i,\bullet)$ and $\phi(f_i,\bullet)$ to a linear combination). Moreover, we can assume that all $f_i$ are orthogonal to each other; otherwise we can add an appropriate linear combination of the $e_i$'s to each $f_j$. We have proved that the elements $(e_1,\ldots,e_k,f_1,\ldots,f_k)$ form a symplectic basis of a submodule $M_1$ of $M$.
	
	We have $M=M_1\oplus M_1^\perp$. Now, by \Cref{coro:free-Smith}, $M_1^\perp$ is also free. Since $M_1^\perp$ is free and symplectic, we can find a symplectic basis $(e_{k+1},\ldots,e_n,f_{k+1},\ldots,f_n)$ of $M_1^\perp$.
	
\end{proof}

Without the assumption "free", there is no such nice form as in \Cref{lem:free-isotr} for an arbitrary isotropic submodule. However, this can be fixed at least for Lagrangians, as is done in \cite[Theorem 7]{albouy2008matrix}.

\begin{prop}
	\label{prop:class-Lagr}
Assume that $L\subset M$ is Lagrangian. Then there exists a symplectic basis $(e_1,\ldots,e_n,f_1,\ldots,f_n)$ of $M$ and positive integers $d_1,d_2,\ldots,d_n$ such that $d_1\mid d_2 \mid d_3 \mid \ldots \mid d_n \mid d\mid d_1^2$, and such that $L=\langle d_1e_1,d_2e_2,\ldots,d_ne_n,\frac{d}{d_1}f_1,\frac{d}{d_2}f_2,\ldots,\frac{d}{d_n}f_n\rangle$.
\end{prop}

\begin{rk}
	By \Cref{lem:free-isotr}, if the Lagrangian $L$ is free, then in the proposition above we can take $d_1=d_2=\ldots=d_n=d$ and then $(f_1,\ldots,f_n)$ is a basis of $L$.
\end{rk}

For the elementary free symplectic $\Zd$-module $S_d=\Zd\oplus \Zd$, it is easy to describe the Lagrangians by hand.

\begin{lem}
	\label{lem:isotr-in-Sd}
	Every isotropic submodule $N$ in $S_d$ is of the form $N=\langle ae,bf\rangle$ for some symplectic basis $(e,f)$ of $S_d$ and some positive integers $a,b$ such that $d\mid ab$ and $b\mid a$. In particular, each Lagrangian submodule $L$ of $S_d$ is of the form $L=\langle ae,\frac{d}{a}f\rangle$, where $a$ is a positive integer such that $a\mid d\mid a^2$.
\end{lem}

\begin{proof}
	Let $m$ be an element of maximal order in $N$ and let $c$ be its order. Write $d=ac$. Since the module $S_d$ is free, we can write $m=ae$ for some $e\in S_d$. By the same argument as in the proof of \Cref{lem:structure-sympl}, we can construct an element $f\in S_d$ such that $\phi(e,f)=1$. Then $(e,f)$ is a symplectic basis of $S_d$.
	
	Let $b$ be the minimal positive integer such that $bf\in N$. We clearly have $d\mid ab$ and $b\mid a$. Let us show that we have $N=\langle ae,bf\rangle$.
	
	Assume that we have $xe+yf\in N$ with some integers $x,y$. Since the order of $xe+yf$ divides $c$, we have $c(xe+yf)=0$ and thus $d\mid cx$, which implies $a\mid x$. Then we have $xe\in N$ and hence $yf\in N$, which implies $b\mid y$. Thus we have $xe+yf\in \langle ae,bf\rangle$.
\end{proof}

The proof of \Cref{prop:class-Lagr} in \cite{albouy2008matrix} is algorithmic. For consistency reasons, we prefer to reformulate it in a different form and to give a more direct module-theoretic proof.

\begin{prop}
	\label{prop:classif-Lang}
	Assume that $L\subset M$ is Lagrangian. Then we can choose a decomposition $M=M_1\oplus M_2\oplus \ldots \oplus M_n$ of $M$ into an orthogonal direct sum of $2$-dimensional free symplectic submodules (i.e., each factor is isomorphic to $S_d$) such that for each $r\in [1;n]$ the submodule $L\cap M_r\subset M_r$ is Lagrangian in $M_r$ (and then automatically for cardinality reasons we have $L=\oplus_{r=1}^n L\cap M_r$). 
\end{prop}

\begin{proof}
	Let $m$ be an element of maximal order in $L$ and let $a$ be its order. Write $d=ab$. Since the module $M$ is free, we can write $m=be$ for some $e\in M$. By the same argument as in the proof of \Cref{lem:structure-sympl}, we can construct an element $f\in M$ such that $\phi(e,f)=1$.
	
	Since $\phi(e,f)=1$, the elements $e$ and $f$ are clearly $\Zd$-linearly independent. Then the submodule $N=\langle e,f\rangle$ of $M$ generated by $e$ and $f$ is free and symplectic. 
	
	We claim that we have $af\in L$. Indeed, since the orders of the elements of $L$ are divisors of $a$, for each $l\in L$ we have $\phi(f,l)\in b\bbZ/d\bbZ\subset \bbZ/d\bbZ=\Zd$ and then we have $\phi(af,l)=0$ for each $l\in L$. This shows $af\in L^\perp=L$.
	
	We see that $L\cap N$ contains $\langle be,af\rangle$, which is a Lagrangian of $N$. Then we must have an equality $L\cap N=\langle be,af\rangle$ because $L\cap N$ is isotropic. 
	
	Now we claim that we have $L=(L\cap N)\oplus (L\cap N^\perp)$. Take $l\in L$ and consider the decomposition $l=l_1+l_2$ such that $l_1\in N$ and $l_2\in N^\perp$. Since $L\cap N$ is a Lagrangian in $N$, we have $(L\cap N)^\perp \cap N=L\cap N$. On the other hand, we know that $l_1=l-l_2\in (L\cap N)^\perp\cap N$, and then $l_1$ must be an element of $L$. Then $l_2=l-l_1$ must also be an element of $L$. Hence we have $l=l_1+l_2$ with $l_1\in N\cap L$ and $l_2\in L\cap N^\perp$. This shows $L=(L\cap N)\oplus (L\cap N^\perp)$.
	
	The submodule $L\cap N^\perp\subset N^\perp$ is Lagrangian for cardinality reasons. Now it remains to apply induction to the free symplectic module $N^\perp$ with the Lagrangian $L\cap N^\perp$.  
\end{proof}

\bigskip
\addtocontents{toc}{{\textbf{Part I: The free case}}}
\centerline{\LARGE Part I: The free case}

\section{Free stabiliser formalism for qudits}
\label{sec:free-stab}

When we work with abelian groups, we usually view them as $\bbZ$-modules or as $\Zd$-modules. In this case, we use additive notation and denote the neutral element by $0$. When we work with general groups (typically non-commutative), we use multiplicative notation and denote the neutral element by $1$. We may also write $I$ instead of $1$ when the identity element is the identity linear operator. For example, the neutral element of the Pauli group $\calP_n$ is denoted by $I$.

For example, the Pauli group $\calP_n$ is not commutative, so we use multiplicative notation for it. However, the quotient $\calP_n/\langle\zeta\rangle$ of the Pauli group by its center is commutative. We will identify it with a (symplectic) $\Zd$-module $\sM$, and we use additive notation for $\sM$. For a commutative subgroup $H\subset \calP_n$ without non-trivial scalars, the subgroup $H$ is isomorphic to its image in $\sM$. However, we use multiplicative notation for $H$ itself and additive notation for its image in $\sM$.

\subsection{Qudits}
\label{subs:qudits}
Qubits are described mathematically by the vector space $\bbC^2$ and two matrices acting on it:
$$
X=
\begin{pmatrix}
0 & 1 \\
1 & 0
\end{pmatrix},
\qquad
Z=
\begin{pmatrix}
1 & 0 \\
0 & -1
\end{pmatrix}.
$$
They satisfy $X^2=Z^2=I$ and $XZ=-ZX$. Note that $(I,X,Z,XZ)$ is a basis of $\End(\bbC^2)$.

Now, qudits generalise this notion. Let $d$ be a positive integer (we recover qubits when $d=2$). Let us write 
$$\bar d=
\begin{cases}
d &\mbox{ if $d$ is odd},\\
2d &\mbox{ if $d$ is even}.
\end{cases}
$$
Let $\zeta$ be a $\dbar$-th primitive root of $1$ and set $\xi=\zeta^2$. Then $\xi$ is a primitive $d$-th root of $1$.
Consider the vector space $\bbC^d$ with the following two matrices acting on it:
$$
X=
\begin{pmatrix}
0 & 0 & \cdots & 0 & 1 \\
1 & 0 & \cdots & 0 & 0 \\
0 & 1 & \cdots & 0 & 0 \\
\vdots &  & \ddots &  & \vdots \\
0 & 0 & \cdots & 1 & 0
\end{pmatrix},
\qquad
Z=
\begin{pmatrix}
1 & 0 & 0 & \cdots & 0 \\
0 & \xi & 0 & \cdots & 0 \\
0 & 0 & \xi^2 & \cdots & 0 \\
\vdots &  &  & \ddots & \vdots \\
0 & 0 & 0 & \cdots & \xi^{d-1}
\end{pmatrix}.
$$

It is convenient to think about this in the following way. The vector space $\bbC^d$ has a basis $\{v_i,i\in\Zd\}$ and we have $X(v_i)=v_{i+1}$ and $Z(v_i)=\xi^i v_i$. Then it is easy to see that they satisfy $X^d=Z^d=I$ and $ZX=\xi XZ$. Moreover, the set $\{X^aZ^b \mid a,b\in \Zd\}$ is a basis of $\End(\bbC^d)$.

\subsection{Qudit Pauli group}
\label{subs:Pauli}

\begin{df}
	The \emph{Pauli group} for one qudit is the subgroup $\calP_1\subset GL(\bbC^d)$ generated by $X$ and $Z$ and by the operator acting by multiplication by $\zeta$. (Abusing the notation, we will also denote this element of the Pauli group by $\zeta$.)
\end{df}

The notation for $\zeta$ may however be confusing when we consider the group algebra $\bbC\calP_1$. In this case we may write $\zeta_{\calP_1}$ and $\zeta_\bbC$ with subscripts to distinguish between the element in the Pauli group and the element of the field $\bbC$.

\begin{rk}
	The operator $ZXZ^{-1}X^{-1}$ acts on $\bbC^d$ by multiplication by $\xi$. So, it may seem more natural to define the Pauli group with only $\xi$ and without $\zeta$. One of the reasons why $\zeta$ is usually added is that without it the surjectivity in \Cref{prop:surj} would fail, see \Cref{rk:surj-fails}. Another reason for doing that will become clear later in this paper when we look at the Pauli group from the point of view of Heisenberg extensions. Without adding $\zeta$, the Pauli group would not satisfy $(c)$ in \Cref{def:Heis}.
\end{rk}

\begin{rk}
	The basis statement at the end of \Cref{subs:qudits} can be strengthened to the following isomorphism of algebras:
	$$
	\bbC \calP_1/(\zeta_{\calP_1}-\zeta_{\bbC})\simeq\End(\bbC^d).
	$$
\end{rk}

Let $n$ be a positive integer. We now give a generalisation of this notion for $n$ qudits. Set $V=(\bbC^d)^{\otimes n}$. For $r\in [1;n]$ denote by $X_r$ and $Z_r$ the operators $X$ and $Z$ acting on the $r$-th qudit, i.e., in other words we have 
$$
X_r=I^{\otimes {r-1}}\otimes X\otimes I^{\otimes {n-r}}, \quad Z_r=I^{\otimes {r-1}}\otimes Z\otimes I^{\otimes {n-r}}.
$$ 

\begin{df}
	The \emph{Pauli group} (for $n$ qudits) is the subgroup $\calP_n\subset GL((\bbC^d)^{\otimes n})$ generated by $X_r$ and $Z_r$ for $r\in [1;n]$ and by the operator acting by multiplication by $\zeta$. (Abusing the notation, we also denote this operator by $\zeta$.)
\end{df}

\begin{rk}
	\label{rk:els-of-Pn}
	Since we have $\End((\bbC^d)^{\otimes n})\simeq \End(\bbC^d)^{\otimes n}$, the algebra $\End((\bbC^d)^{\otimes n})$ has the following basis
	$$
	\{X_1^{a_1}Z_1^{b_1}\ldots X_n^{a_n}Z_n^{b_n},\, a_r,b_r\in [0;d-1]\}.
	$$
	On the other hand, we have 
	\begin{equation}
	\label{eq:Pauli-elements}
	\calP_n=\{\zeta^c X_1^{a_1}Z_1^{b_1}\ldots X_n^{a_n}Z_n^{b_n},\, a_r,b_r\in [0;d-1],c\in[0;\dbar-1]\}.
	\end{equation}
	This shows immediately that we have an isomorphism of algebras 
	$$
	\bbC\calP_n/(\zeta_{\calP_n}-\zeta_{\bbC})\simeq \End((\bbC^d)^{\otimes n}).
	$$
\end{rk}

\begin{coro}
	\label{coro:Pn-gen+rel}
	The group $\calP_n$ can be defined by generators and relations in the following way. The generators are $Z_1,\ldots,Z_n,X_1,\ldots,X_n,\zeta$ and the relations are
	$$
	\zeta^{\dbar}=1,\qquad X_k\zeta=\zeta X_k,\qquad Z_k\zeta=\zeta Z_k,\qquad X_kX_r=X_rX_k,\qquad Z_kZ_r=Z_rZ_k,
	$$
	$$
	X_kZ_r=Z_rX_k \mbox{ \rm(if $r\ne k$)}, \qquad X_kZ_k=\zeta^2Z_kX_k,\qquad X_k^d=1, \qquad Z_k^d=1.
	$$
\end{coro}
\begin{proof}
	Denote by $\calP'_n$ the group generated by the generators and relations above. Since all these relations are satisfied in the Pauli group, we have a group homomorphism $\calP'_n\to \calP_n$ sending the generators to the elements with the same names. Since $\calP_n$ is generated by $Z_1,\ldots,Z_n,X_1,\ldots,X_n,\zeta$, this homomorphism is surjective. 
	
	Let us check the injectivity. It is clear that, using the relations, we can write each element of $\calP'_n$ as $\zeta^c X_1^{a_1}Z_1^{b_1}\ldots X_n^{a_n}Z_n^{b_n}$ with $a_r,b_r\in [0;d-1]$ and $c\in[0;\dbar-1]$. However, by \Cref{rk:els-of-Pn}, all these elements go to different elements of $\calP_n$. This implies the injectivity.
	
\end{proof}

\subsection{Qudit Clifford group}

\begin{df}
	The \emph{qudit Clifford group} $Cl_n$ is the normaliser of the Pauli group $\calP_n$ in the unitary group $U((\bbC^d)^{\otimes n})$ quotiented by scalars.
\end{df}

\begin{rk}
	We give below the definition of the Clifford group more common in the quantum information literature. However, as explained in \Cref{subs:descr-Heis-symp}, the Clifford group $Cl_n$ is nothing else than the group $\Aut_\zeta(\calP_n)$ of automorphisms of $\calP_n$ that fix $\zeta$.
\end{rk}

Let $\sM$ be a free $2n$-dimensional $\Zd$-module with a basis $(z_1,\ldots,z_n,x_1,\ldots,x_n)$.
Any two elements of the Pauli group $\calP_n$ commute up to multiplication by a power of $\xi$, this means that the group becomes commutative when we quotient it by the subgroup $\langle\zeta\rangle$ generated by $\zeta$. Moreover, the order of each element in the quotient is a divisor of $d$, so the quotient can be considered as a $\Zd$-module. We get the following obvious lemma.

\begin{lem}
	\label{lem:isom-quot-Pauli}
	We have the following isomorphism of $\Zd$-modules
	$$
	\calP_n/\langle\zeta\rangle\simeq\sM,\qquad Z_r\mapsto z_r,\, X_r\mapsto x_r.
	$$
\end{lem}
Denote by $\tau$ the natural homomorphism $\tau\colon \calP_n \to \sM$ defined as the quotient map $\calP_n \to \calP_n/\langle \zeta \rangle$ composed with the isomorphism above.
Consider a symplectic form $\phi\colon \sM\times \sM\to  \Zd$ given by
$$
\phi(z_i,z_j)=\phi(x_i,x_j)=0,\qquad \phi(z_i,x_j)=-\phi(x_j,z_i)=\delta_{i,j},
$$
where $\delta_{i,j}$ is the Kronecker delta. The following lemma is straightforward.

\begin{lem}
	\label{lem:furmula-commut}
	For $x,y\in\calP_n$ we have $xy=\xi^{\phi(\tau(x),\tau(y))}yx$.
\end{lem} 
The lemma above shows us that, even if the quotient of $\calP_n$ by $\langle\zeta\rangle$ makes the group commutative, the symplectic form still keeps all the information about commutation relations before taking the quotient.

Let $\Phi$ be an automorphism of $\calP_n$ coming from conjugation by an element of the Clifford group. If we have $xy=\xi^a yx$ for $x,y\in \calP_n$ then we also have $\Phi(x)\Phi(y)=\xi^a\Phi(y)\Phi(x)$. This implies that, in view of \Cref{lem:furmula-commut}, the action of $Cl_n$ on $\calP_n/\langle\zeta\rangle\simeq \sM$ by conjugation preserves the symplectic form $\phi$. In particular we get a homomorphism of groups $Cl_n\to Sp(\sM,\phi)$.

\begin{prop}
	\label{prop:surj}
	The homomorphism $Cl_n\to  Sp(\sM,\phi)$ above is surjective.
\end{prop}
\begin{proof}
	This follows from \cite[\S II\,C]{hostens2005stabilizer}. More precisely, they claim that each $C\in Sp(\sM,\phi)$ and $h$ satisfying condition \cite[(10)]{hostens2005stabilizer} can be lifted to an element of $Cl_n$ whose image in $Sp(\sM,\phi)$ is $C$. The condition on $h$ is just a parity condition, so for each $C$ there is such an element $h$.
	
	An alternative proof of this fact in a more general setting will be given in \Cref{sec:not-free} (see \Cref{rk:proof-surj-Cl-Sp}).
\end{proof}

\begin{rk}
	\label{rk:surj-fails}
	The Pauli group was defined with an extra element $\zeta$. Let us explain why the surjectivity in the proposition above fails if we do not add it. Take $n=1$ and $d=2$. Imagine that we have defined the Pauli group without adding $\zeta$ (which is $i\in \bbC$ in the case $d=2$) and suppose that we have defined the Clifford group in terms of the smaller Pauli group. Then the Pauli group contains $\pm ZX$ (which are of order $4$) but not $\pm i ZX$ (which are of order $2$). The symplectic map in $Sp(F_1,\phi)$ given by $z_1\mapsto z_1+x_1$, $x_1\mapsto x_1$ is not in the image of the modified Clifford group because it is impossible to conjugate $X$ and get $\pm ZX$ for order reasons.
\end{rk}

\subsection{Free stabiliser formalism}
\label{subs:stab-formal}

For a subgroup $H\subset \calP_n$, let $V^H=\{v\in V|\, \forall h\in H,\, h(v)=v\}$ be the stabiliser subspace of $H$ in $V$. Denote by $N(H)$ the normaliser of $H$ in $\calP_n$.

\begin{thm}
	\label{thm:main-free-dstab}
	Let $H$ be an abelian subgroup of $\calP_n$ satisfying the following two conditions
	\begin{equation}
	\label{eq:cond-H-1}
	H\cap \{\zeta^a, a\in \bbZ\}=\{I\}.
	\end{equation}
	\begin{equation}
	\label{eq:cond-H-2}
	H\simeq \Zd^k \mbox{ for }k\in \bbZ_{\geqslant 0}.
	\end{equation}
	Then we have the following.
	
	We have $k\leqslant n$ and there is a unitary isomorphism of $V^H$ with $(\bbC^d)^{\otimes n-k}$ such that the action of $N(H)/H$ is identified with the action of $\calP_{n-k}$ on $(\bbC^d)^{\otimes n-k}$.
\end{thm}

Before proving the theorem, we start with some preparation.

\begin{rk}
	\label{rk:norm=centr}
	For each commutative subgroup $H\subset \calP_n$ satisfying \eqref{eq:cond-H-1}, its normaliser is the same as its centraliser. 
	Moreover, for $x\in\calP_n$, the centraliser of $xHx^{-1}$ in $\calP_n$ must be $xN(H)x^{-1}$. But since the elements of $xHx^{-1}$ and $H$ are the same up to powers of $\xi$, the centralisers must be the same. Hence we have $xN(H)x^{-1}=N(H)$. This means that $N(H)$ is normal in $\calP_n$.
	
Moreover, since the kernel of the map $\tau\colon\calP_n\to \sM$ is given by scalars, the composition $H\to \calP_n\to \sM$ is injective by \eqref{eq:cond-H-1}. Then $H$ is isomorphic to $\tau(H)$. This implies in particular that $H$ may only contain elements whose orders divide $d$. (While $\calP_n$ itself contains elements of order $\dbar$.) This allows to see $H$ as a $\Zd$-module.
\end{rk}

\begin{prop}
	\label{prop:conj-Zs}
	If the subgroup $H\subset \calP_n$ satisfies the conditions in \Cref{thm:main-free-dstab}, then $k\leqslant n$ and $H$ is conjugate via an element of the Clifford group $Cl_n$ to the subgroup of $\calP_n$ generated by $Z_1,\dots,Z_k$.
\end{prop}

\begin{proof}

	By \eqref{eq:cond-H-1}, the group $H$ is isomorphic to its image $\tau(H)$ (see \Cref{rk:norm=centr}). Next, condition \eqref{eq:cond-H-2} implies that $\tau(H)$ is a free $\Zd$-submodule of $\sM$ of rank $k$. Moreover, by \Cref{lem:furmula-commut}, since $H$ is commutative, the submodule $\tau(H)\subset \sM$ is isotropic in the sense of \Cref{df:isotr}.
	
	Now, by \Cref{lem:free-isotr}, we have $k\leqslant n$ and we can find an element $\beta\in Sp(\sM,\phi)$ such that $\beta(\tau(H))=\langle z_1,z_2,\ldots,z_k\rangle$. By \Cref{prop:surj}, the element $\beta\in Sp(\sM,\phi)$ can be lifted to an element $\hat\beta\in Cl_n$ and we get $\tau(\hat\beta H \hat\beta^{-1})=\langle z_1,z_2,\ldots,z_k\rangle$. This means that $\hat\beta H \hat\beta^{-1}$ has a set of generators of the form $\{\xi^{a_1}Z_1,\xi^{a_2}Z_2,\ldots,\xi^{a_k}Z_k\}$ for some integers $a_1,\ldots,a_k$. Now we can additionally conjugate the obtained subgroup by $X_1^{a_1}X_2^{a_2}\ldots X_k^{a_k}$ to get the subgroup generated by $Z_1,Z_2,\ldots,Z_k$.
\end{proof}

\begin{rk}
The proposition above may give the impression that all codes based on free groups $H$ are trivial. However, one should take into account the following point. When studying the error-correction properties of a code, it is important to consider the length function $\ell\colon \calP_n \to \bbZ_{\geqslant 0}$ defined as the number of indices $i$ such that $a_i$ or $b_i$ is nonzero in the presentation \eqref{eq:Pauli-elements}. This length function is not invariant under the action of $Cl_n$, and therefore conjugating $H$ to a subgroup generated by $Z$-operators does not preserve the error-correction properties of the code. In order for the code to correct many errors, one requires that the elements of $N(H)\setminus H$ have sufficiently large length.
\end{rk}

\begin{proof}[Proof of \Cref{thm:main-free-dstab}]
	In view of \Cref{prop:conj-Zs}, it is enough to prove the theorem in the case when $H$ is generated by $Z_1,\ldots,Z_k$. Since the normaliser $N(H)$ of $H$ in the Pauli group coincides with the centraliser (see \Cref{rk:norm=centr}), by \Cref{lem:furmula-commut} we have $x\in N(H)$ if and only if $\tau(x)\in \langle z_1,z_2,\ldots,z_k\rangle^\perp=\langle z_1,z_2,\ldots,z_n,x_{k+1},\ldots,x_n\rangle$. This means that $N(H)$ is generated by $\zeta$, $Z_1,\ldots,Z_n$, $X_{k+1},\ldots,X_n$.
	
	The vector space $V^H$ is the subspace $v_0^{\otimes k}\otimes (\bbC^d)^{\otimes n-k}\subset (\bbC^d)^{\otimes n}$ and has an obvious identification with $(\bbC^d)^{\otimes n-k}$. The action of $N(H)/H$ on $V^H$ is generated by $\zeta,Z_{k+1},\ldots,Z_n,X_{k+1},\ldots,X_n$. Under this identification, they are exactly the generators of the Pauli group $\calP_{n-k}$ acting on $(\bbC^d)^{\otimes n-k}$.

\end{proof}

\subsection{Eigenspaces}
\label{subs:eigen-free}

Let $H\subset \calP_n$ be as in \Cref{thm:main-free-dstab}. Since the orders of elements of $H$ divide $d$, we may identify the groups $\Hom(H,\Zd)$, $\Hom(H,\mu_d)$ and $\Hom(H,\bbC^*)$ and we denote each of them by $H^*$.
Let $V=\bigoplus_{\chi\in H^*}V_\chi$ be the decomposition of $V$ into a direct sum of common eigenspaces of $H$, where for $\chi\in H^*=\Hom(H,\bbC^*)$ we set $V_\chi=\{v\in V|\, \forall h\in H,\, h\cdot v=\chi(h)v \}$.

Then for each $x\in \calP_n$ and $\chi\in H^*$ we have $x(V_\chi)=V_{\chi'}$ for some $\chi'\in H^*$. This yields an action of $\calP_n$ on $H^*$. Moreover, the protected space $V^H$ is the same thing as $V_{\mathbb 1}$, where $\mathbb 1\in H^*$ is the trivial character.

\begin{rk}
	\label{rk:trans=dim}
	It is clear that all eigenspaces $V_\chi$ for $\chi$'s in one $\calP_n$-orbit have the same dimension. Then having $\dim V^H=\frac{\dim V}{\# H^*}$ is equivalent to the transitivity of the $\calP_n$-action on $H^*$. 
\end{rk}

\begin{prop}
	\label{prop:trans-free}
	The action of $\calP_n$ on $H^*$ is transitive. In particular all vector spaces $V_\chi$ have the same dimension.
\end{prop}
\begin{proof}
	By \Cref{rk:trans=dim}, it is enough to check $\dim V^H=\frac{\dim V}{\# H^*}$. This is true because by \Cref{thm:main-free-dstab}, the dimension of $V^H$ is $d^{n-k}$.
\end{proof}

\begin{rk}
	We deduced above the transitivity from the dimension of the protected space. However, we will see in \Cref{subs:descr-Heis-symp} that the transitivity is still true without the assumption "free". However, there we invert the logic of \Cref{rk:trans=dim}. We first deduce the transitivity using the theory of Heisenberg extensions, then we deduce the dimension of the protected space from the transitivity. 
\end{rk}

By the proposition above, each eigenspace $V_\chi$ is of the form 
$$V_\chi = x(V^H) = V^{xHx^{-1}}$$
for some $x \in \calP_n$. 
Moreover, the stabiliser of $\mathbb 1 \in H^*$ in $\calP_n$ is $N(H)$. 
Since $V_\chi = x(V_{\mathbb 1})$, the stabiliser of $\chi$ is 
$xN(H)x^{-1} = N(H)$ because $N(H)$ is normal. 
Thus, for each eigenspace $V_\chi$, the set of elements of the Pauli group preserving it is precisely $N(H)$. 

\begin{lem}
	Assume $\chi \in H^*$ and $x \in \calP_n$. Then we have the following equivalences:
	$$
	x(V_\chi) \subset V_\chi 
	\Leftrightarrow x(V^H) \subset V^H 
	\Leftrightarrow x \in N(H).
	$$
\end{lem}

\section{Kitaev model for qudits}
\label{sec:Kitaev}

\subsection{CSS codes}
\label{subs:CSS}

The Pauli group has two special abelian subgroups: the subgroup $\calP^Z_n$ generated by $Z_1,Z_2,\ldots,Z_n$ and the subgroup $\calP^X_n$ generated by $X_1,X_2,\ldots,X_n$. Moreover, the homomorphism $\tau\colon \calP_n\to\sM$ induces isomorphisms of $\calP^Z_n$ and $\calP^X_n$ with the Lagrangian submodules $\langle z_1,z_2,\ldots,z_n \rangle$ and $\langle x_1,x_2,\ldots,x_n \rangle$ of $\sM$.

When we speak about CSS codes, we mean that we take an abelian subgroup $H\subset \calP_n$ of a particular form: $H=H^Z\cdot H^X$ where $H^Z$ is a subgroup of $\calP^Z_n$ and $H^X$ is a subgroup of $\calP^X_n$ such that $H^Z$ and $H^X$ commute. Moreover, it is clear that choosing $H^Z$ and $H^X$ in this way is equivalent to choosing submodules of $\langle z_1,z_2,\ldots,z_n \rangle$ and $\langle x_1,x_2,\ldots,x_n \rangle$ that are orthogonal to each other.

The following lemma is obvious.

\begin{lem}
	The group $H=H^Z\cdot H^X$ satisfies automatically \eqref{eq:cond-H-1}. Moreover, if $H^X$ and $H^Z$ satisfy \eqref{eq:cond-H-2}, then $H$ satisfies \eqref{eq:cond-H-2}.
\end{lem}

In particular this means that if $H^X$ and $H^Z$ satisfy \eqref{eq:cond-H-2}, then we can apply \Cref{thm:main-free-dstab} to $H$.

\subsection{Construction of the model}
\label{subs:Kitaev-mod}

We describe here the qudit version of the Kitaev model \cite{kitaev2003fault}. It can be viewed as the special case $G=\bbZ/d\bbZ$ of the general construction in \cite[\S 5]{kitaev2003fault}, and it recovers the original qubit model of \cite[\S 2]{kitaev2003fault} when $d=2$. We note, however, that the passage to $G=\bbZ/d\bbZ$ requires fixing specific conventions, which we choose so as to make the model compatible with our algebraic framework.

Let $\Gamma=(S,E)$ be a finite oriented graph, where $S$ is the set of vertices and $E$ is the set of arrows. Let $\Sigma=\Sigma_g$ be a connected orientable compact surface of genus $g\in \bbZ_{\geqslant 0}$. Assume that the graph $\Gamma$ is inscribed in $\Sigma$ (i.e., the surface $\Sigma$ is a CW-complex whose $1$-skeleton is $\Gamma$). Denote by $F$ the set of faces (i.e., the connected components of $\Sigma\backslash \Gamma$). Note that the graph $\Gamma$ is connected by construction.

We place a qudit on each arrow of the graph. The Hilbert space describing this system is
$
V=(\bbC^d)^{\otimes \#E}.
$
For each vertex $s\in S$ and each face $f\in F$ we define a vertex operator $A_s\colon V\to V$ and a face operator $B_f\colon V\to V$ as follows.

The vertex operator $A_s$ acts by $X$ on each qudit sitting on an arrow entering the vertex $s$, and by $X^{-1}$ on each qudit sitting on an arrow leaving $s$.  
The face operator $B_f$ acts by $Z$ on each qudit sitting on an arrow for which the face $f$ lies on the right of the oriented arrow, and by $Z^{-1}$ on each qudit sitting on an arrow for which the face $f$ lies on the left of the oriented arrow.

\begin{tikzpicture}[>=Stealth,baseline=(current bounding box.center)]
\draw[->,thick] (0,0) -- (3,0);

\node at (-0.5,0) {$X^{-1}$};
\node at (3.3,0) {$X$};

\node at (1.5,0.35) {$Z^{-1}$};       
\node at (1.5,-0.35) {$Z$};  
\end{tikzpicture}

\begin{rk}
	Note that there is an involution $\psi\colon \bbC^d\to\bbC^d$, $v_i\mapsto v_{-i}$. It satisfies $X\psi=\psi X^{-1}$ and $Z\psi=\psi Z^{-1}$. This implies that the model does not change when we reverse an arrow in the graph and simultaneously apply $\psi$ to the corresponding qudit.
\end{rk}

The $(\#S+\#F)$ operators $A_s$ for $s\in S$ and $B_f$ for $f\in F$ commute. They also satisfy
$$
\prod_{s\in S}A_s=I,\qquad 
\prod_{f\in F}B_f=I,\qquad 
A_s^d=B_f^d=I.
$$

Now take $n=\#E$ and consider the subgroup $H$ of the Pauli group generated by the operators $A_s$ for $s\in S$ and $B_f$ for $f\in F$.

\begin{lem}
	The subgroup $H$ defined above satisfies condition $\eqref{eq:cond-H-2}$ with $k=\#S+\#F-2$. 
\end{lem}

\begin{proof}
	The group $H$ is clearly of the form $H=H^Z\cdot H^X$ as in \Cref{subs:CSS}. It is enough to check that $H^Z$ is a free $\Zd$-module of rank $\#F-1$ and $H^X$ is a free $\Zd$-module of rank $\#S-1$.
	
	Let us show that if we remove one $A_{s_0}$ for some $s_0\in S$, then the remaining operators $\{A_s,s\in S\backslash\{s_0\}\}$ form a basis of the $\Zd$-module $H^X$.
	Assume that we have a relation of the form $\prod_{s\in S}A_s^{a_s}=I$ for some integers $a_s$. The product on the left-hand side acts on the qudit on the arrow $e\in E$ by $X^{a_{s_1}-a_{s_2}}$ when the arrow goes from $s_1$ to $s_2$. This shows that, since the graph is connected, the operator $\prod_{s\in S}A_s^{a_s}$ may be the identity only when all $a_s$ are congruent modulo $d$. This shows that once we remove $A_{s_0}$, the remaining $A_s$ have no relations over $\Zd$.
	
	The same argument holds for the $B_f$ to show that $H^Z$ is a free $\Zd$-module of rank $\#F-1$.
\end{proof}

\subsection{Protected space}

Consider the protected space $V^H$. The theory developed in \Cref{subs:stab-formal} is applicable to it. Its dimension could be deduced from \Cref{thm:main-free-dstab}: it is equal to 
$$
d^{n-k}=d^{\#E-\#S-\#F+2}=d^{2g}. 
$$
However, it could be computed without using the stabiliser formalism, via an explicit homological argument below (which is close to the argument in \cite{freedman2001projective}).

\begin{prop}
	We have $\dim V^H=d^{2g}$.
\end{prop}
\begin{proof}
	Consider the cellular cochain complex
	$$
	\calC^0 \xrightarrow{d^0} \calC^1 \xrightarrow{d^1} \calC^2
	$$
	for $\Sigma$ over $\Zd$.
	We have 
	\begin{itemize}
		\item $\calC^0=$ $\Zd$-linear combinations of elements of $S$,
		\item $\calC^1=$ $\Zd$-linear combinations of elements of $E$,
		\item $\calC^2=$ $\Zd$-linear combinations of elements of $F$.
	\end{itemize}
	
	The most obvious basis of $V$ is given by $\calC^1$. The basis of $\bigcap_{f\in F}\Ker(B_f-I)\subset V$, which is a subset of this basis, is given by $Z^1\subset \calC^1$, where $Z^1=\Ker(d^1)$. Finally, the conditions coming from the $A_s$ say that some basis elements must have the same coefficients. Then $V^H$ has a basis given by $Z^1/B^1=H^1(\Sigma,\Zd)\simeq \Zd^{2g}$, where $B^1=\Im (d^0)$.
\end{proof}

\begin{rk}
	The dimension of $V^H=V_{\mathbb 1}$ for the Kitaev model could also be computed in a much more elementary computational way. We could see the transitivity of the action of $\calP_n$ on the set of characters $\chi\in H^*$ without using the stabiliser formalism. Indeed, this transitivity becomes obvious when we describe the characters $\chi\in H^*$ as configurations of charges (see \Cref{subs:eigen-Kitaev}) and we look at how path operators (see \Cref{subs:path}) act on the set of characters. Once we know that all eigenspaces have the same dimension, we deduce the dimension of the protected space $V^H$ as $\frac{\dim V}{\# H^*}$.
\end{rk}

\subsection{Description of the eigenspaces}
\label{subs:eigen-Kitaev}

We have a decomposition into a direct sum of common eigenspaces: $V=\bigoplus_{\chi\in H^*} V_\chi$. The protected space $V^H$ is one of the $V_\chi$.

Let us give a combinatorial description of each $\chi$. Let us imagine that each vertex $s\in S$ may contain a number of copies of elementary electric charges $e$ (with the convention that having $d$ charges $e$ is the same as having no charge at all, so the charges are considered modulo $d$). Similarly, we imagine that each face may contain a number of copies of magnetic charges $m$ (also modulo $d$). 

Now we interpret each $\chi$ as a distribution of charges. Fix $\chi$, and take $v\in V_\chi$, $v\ne 0$.
If we have $A_sv=\xi^p v$ and $B_fv=\xi^q v$, then we say that there are $p$ elementary electric charges $e$ at $s$ (or simply the charge $p\cdot e$ at $s$), and that there are $q$ elementary magnetic charges $m$ at $f$ (or simply the charge $q\cdot m$ at $f$). Additionally, the configuration of charges must be such that the total electric and magnetic charge is zero modulo $d$.

\subsection{Path operators}
\label{subs:path}

For each path $t$ in the graph from a vertex $s_1$ to a vertex $s_2$ (the path is oriented from $s_1$ to $s_2$, but we do not require this orientation to coincide with the orientation of the graph chosen before), we construct a path operator $S^Z(t)\colon V\to V$ as follows. We apply $Z$ to all qudits along the path if the chosen orientation of the path agrees with the orientation of the corresponding edge of the graph, and we apply $Z^{-1}$ otherwise.

Similarly, for each dual path $t$ (that is, a path in the dual graph, where faces and vertices are exchanged) from a face $f_1$ to a face $f_2$, we construct a path operator $S^X(t)\colon V\to V$ as follows. We apply $X$ to all qudits on edges crossed by the path whose orientations point to the right of the path, and $X^{-1}$ to those pointing to the left.

Let us see how these operators modify eigenspaces. Note that $S^Z(t)$ commutes with all $A_s$ and $B_f$ except for $A_{s_1}$ and $A_{s_2}$. This means that when we have $S^Z(t)\cdot V_\chi\subset V_{\chi'}$, the configurations $\chi$ and $\chi'$ have the same charges except at $s_1$ and $s_2$. Moreover, we see that the path operator moves a charge $-e$ from $s_1$ to $s_2$.

Similarly, we see that for a dual path $t\colon f_1\to f_2$, the path operator $S^X(t)$ moves a magnetic charge $-m$ from $f_1$ to $f_2$.

\begin{rk}
	Since path operators act transitively on the set $H^*$ of $\chi$'s, all vector spaces $V_\chi$ have the same dimension. This allows another computation of the dimension of $V^H$. The total dimension of $V$ is $d^{\# E}$ and the number of possible $\chi$'s is $d^{\# S+\# F-2}$. Then the dimension of each $V_\chi$ is $d^{\# E-\# S-\# F+2}=d^{2g}$.

	If $\Sigma$ is a sphere ($g=0$), then each $V_\chi$ is one-dimensional. Physically, this means that the configuration of charges completely determines the physical state of the system. For higher genus ($g>0$), the eigenspaces $V_\chi$ are degenerate (more than one-dimensional). This means that a configuration of charges does not determine the state of the physical system and that there are additional interior degrees of freedom.
\end{rk}  

\begin{rk}
	Take a vector $v\in V^H$. Take two different paths $t,t'\colon s_1\to s_2$ with the same starting and ending points. Then we have $S^Z(t)\cdot v=S^Z(t')\cdot v$ whenever the paths $t$ and $t'$ are homotopic (this is always true for $g=0$). However, for $g>0$, a loop (closed path) which is not homotopic to the identity can send a vector $v\in V^H$ to another element of $V^H$.
	
	Now take a vector $v\in V_\chi$. Then we have $S^Z(t)\cdot v=S^Z(t')\cdot v$ in the case when the paths $t$ and $t'$ are homotopic in such a way that the homotopy does not go through magnetic charges. However, if there is a homotopy traversing magnetic charges, then $S^Z(t)\cdot v$ and $S^Z(t')\cdot v$ differ by a power of $\xi$ depending on the magnetic charge.
	
	In particular, when an $e$-particle is moved clockwise around an $m$-particle, this multiplies the state of the system by the phase factor $\xi$.
\end{rk}

\subsection{The normaliser}
As we have seen in \Cref{subs:eigen-free}, the set of elements of the Pauli group that preserve each eigenspace $V_\chi$ is the same and is equal to $N(H)$. Let us compute it.

\begin{prop}
	The group $N(H)$ is generated by $\zeta$, by the operators $S^Z(t)$ associated with loops (closed paths) $t$ and the operators $S^X(t')$ associated with dual loops (closed dual paths) $t'$.
\end{prop}

\begin{proof} 
	An element $x\in \calP_n$ has the form $\zeta^c\prod_{e\in E}Z_e^{a_e}X_e^{b_e}$ with $c,a_e,b_e\in \bbZ$. The fact that $x$ commutes with the vertex operators $A_s$ imposes conditions on the $a_e$'s, and the fact that $x$ commutes with the face operators $B_f$ imposes conditions on the $b_e$'s. 
	
	Let us first look at the $a_e$'s. They define a $\Zd$-valued flow on the graph $\Gamma$, where the value on the arrow $e$ is $a_e$ modulo $d$. The condition that $x$ commutes with a vertex operator $A_s$ says exactly that at the vertex $s$, the sum of incoming flows is equal to the sum of outgoing flows. In other words, this flow is divergence-free. By a standard fact from graph theory, every divergence-free flow is a sum of flows supported on cycles. This shows that the $Z$-part of $x$ can be written as a product of operators $S^Z(t)$ associated with loops $t$. 
	
	Similarly, the coefficients $b_e$ define a $\Zd$-valued flow on the dual graph, and the condition that $x$ commutes with each face operator $B_f$ says exactly that this dual flow is divergence-free. Hence the $X$-part of $x$ can be written as a product of operators $S^X(t')$ associated with dual loops $t'$. 

\end{proof}

\bigskip
\addtocontents{toc}{{\textbf{Part II: The general case}}}
\centerline{\LARGE Part II: The general case}

\section{Relaxing the freeness assumption}
\label{sec:not-free}

\subsection{Heisenberg extension}

In this section we introduce some notions about Heisenberg groups that are well known. However, we use some specific conventions, since we want the Heisenberg group to be finite and to have a proper relation to Pauli groups (\Cref{lem:Heis=Pauli}). Our definition uses the central extension by the group $\mu_\dbar$ of $\dbar$-th roots of unity. Our axiom $(c)$ in \Cref{def:Heis} is not completely standard. We add it because we want the uniqueness (\Cref{prop-Heis-unique}), the precise presentation by generators and relations (\Cref{lem:lift-give-genrel}), and the proper relation to the Pauli group (\Cref{lem:Heis=Pauli}).

\begin{df}
	For a finite $\Zd$-module $M$, we say that $(y_1,y_2,\ldots,y_t)$ is a quasi-basis of $M$ if $M=\langle y_1,y_2,\ldots,y_t\rangle$ and for each $\lambda_1,\ldots,\lambda_t\in\Zd$ the relation $\lambda_1y_1+\lambda_2y_2+\ldots+\lambda_ty_t=0$ implies $\lambda_1y_1=\lambda_2y_2=\ldots=\lambda_ty_t=0$. (This condition means that the only relations between these elements come from the orders of these elements.)
\end{df}

If $(y_1,y_2,\ldots,y_t)$ is a quasi-basis of $M$ such that $d_k$ is the order of $y_k$, then each element $m\in M$ can be written in a unique way as $m=\sum_{k=1}^t \mu_ky_k$ with $\mu_k\in\bbZ_{d_k}$. Each finite $\Zd$-module has a quasi-basis by \Cref{prop:classif-Zd-mod}.

\begin{df}
	\label{def:Heis}
	Let $M$ be a finitely generated $\Zd$-module. Let $\phi\colon M\times M\to \Zd$ be an alternating (see \Cref{def:alt-form}) $\Zd$-linear form. We say that a group $G$ is a \emph{Heisenberg extension} of $(M,\phi)$ if the following axioms hold.
	\begin{itemize}
		\item[$(a)$] There is a surjective group homomorphism $\tau\colon G\to M$ whose kernel is isomorphic to the group $\mu_{\dbar}$ of $\dbar$-th roots of unity.
		\item[$(b)$] For each $g,h\in G$, we have $gh=\xi^{\phi(\tau(g),\tau(h))}hg$.
		\item[$(c)$] Each element of the abelian group $M$ has a preimage in $G$ having the same order.
	\end{itemize}
\end{df}

Fix a quasi-basis $(y_1,y_2,\ldots,y_t)$ of $M$ and denote by $d_k$ the order of $y_k$. Consider the group $G$ defined by generators and relations in the following way. The generators of $G$ are $Y_1,\ldots,Y_t$, $\zeta$, and the relations are
\begin{equation}
\label{eq:Heis-1}
\quad\zeta^{\dbar}=1,\qquad Y_k\zeta=\zeta Y_k, \qquad Y_kY_r=\zeta^{2\phi(y_k,y_r)}Y_rY_k,
\end{equation}
\begin{equation}
\label{eq:Heis-2}
Y_k^{d_k}=1,
\end{equation}
where $k,r\in [1;t]$.

\begin{lem}
	\label{lem:el-of-Heis}
	Each element of $G$ can be written in a unique way in the form 
	\begin{equation}
	\label{eq:el-of-Heis}
	\zeta^a Y_1^{a_1}\ldots Y_t^{a_t},\qquad a\in[0;\dbar-1], a_k\in [0;d_k-1].
	\end{equation}
\end{lem}	
\begin{proof}
	The fact that each element of $G$ can be written in such a form is easy to see from the relations, but the uniqueness is not obvious (the most delicate point is to see that the element $\zeta\in G$ is still of order $\dbar$ and not of a smaller order). To show the uniqueness, we construct a representation of $G$ on which the elements of this form act by different operators.

	It is clear that for a given alternating form $\phi$ we can find a $\Zd$-bilinear form $\phi_0$ such that for each $m_1,m_2\in M$ we have $\phi(m_1,m_2)=\phi_0(m_1,m_2)-\phi_0(m_2,m_1)$. (We can for example take $\phi_0$ such that $\phi_0(y_k,y_r)=\phi(y_k,y_r)$ if $k<r$ and $\phi_0(y_k,y_r)=0$ if $k>r$.) Consider the complex vector space $\bbC^M$ having a basis $\{v_m,\,m\in M\}$. We consider the operators $Y'_1,\ldots,Y'_t$ on $\bbC^M$ given by 
	$$
	Y'_k\cdot v_m=\xi^{\phi_0(y_k,m)}v_{y_k+m}.
	$$
	and, abusing notation, we denote by $\zeta$ and $\xi$ the operators of multiplication by the complex scalars $\zeta$ and $\xi$ on $\bbC^M$.

	The operators $Y'_1,\ldots,Y'_t,\zeta$ satisfy relations \eqref{eq:Heis-1} but probably not \eqref{eq:Heis-2} because ${Y'}_k^{d_k}$ may be a power of $\zeta$ which is not necessarily equal to $1$. However, we can fix this by multiplying the operators $Y'_k$ by appropriate roots of unity. This fixes relation \eqref{eq:Heis-2} without losing relations \eqref{eq:Heis-1}. We then get a representation of $G$ on the vector space $\bbC^M$. It is obvious from this representation that the elements of the form \eqref{eq:el-of-Heis} act all differently on this representation.
\end{proof}
	
\begin{lem}
	\label{lem:constr-Heis}
The group $G$ above is a Heisenberg extension of $(M,\phi)$.
\end{lem}
\begin{proof}
By \Cref{lem:el-of-Heis}, the element $\zeta\in G$ is of order $\dbar$. So, the group $G$ contains a subgroup isomorphic to $\mu_\dbar$ generated by $\zeta$. The quotient of $G$ by this subgroup, in view of \Cref{lem:el-of-Heis}, is isomorphic to the abelian group $M$. This shows condition $(a)$ in \Cref{def:Heis}. Condition $(b)$ follows immediately from \eqref{eq:Heis-1}. 

Let us show $(c)$. Consider an element $m=\sum_{k=1}^ta_ky_k$ with $a_k\in [0;d_k-1]$. The order of $m$ is the minimal positive integer $a$ such that $d_k\mid aa_k$ for each $k$. Now, we want to find a lift of $m$ in $G$ of order $a$. Set $g=Y_1^{a_1}Y_2^{a_2}\ldots Y_t^{a_t}$. Then $g^a$ is equal to $Y_1^{aa_1}Y_2^{aa_2}\ldots Y_t^{aa_t}=1$ multiplied by 
\begin{equation}
\label{eq:power-of-xi}
\xi^{\sum_{k>r}\frac{a(a-1)}{2}a_ka_r\phi(y_k,y_r)}.
\end{equation} 
Since $y_k$ is of order $d_k$, we have $\phi(y_k,y_r)\in \frac{d}{d_k}\Zd$. Then we have $aa_k\phi(y_k,y_r)\in d_k\frac{d}{d_k}\Zd$, so $aa_k\phi(y_k,y_r)=0$ in $\Zd$. Then if $a$ is odd, $\frac{a-1}{2}$ is an integer and $\frac{a(a-1)}{2}a_ka_r\phi(y_k,y_r)=0$ in $\Zd$, so \eqref{eq:power-of-xi} is $1$ and we get $g^a=1$. However, for $a$ even (this is possible only if $d$ is even), the number $\frac{a-1}{2}$ is a half-integer, so we get $\frac{a(a-1)}{2}a_ka_r\phi(y_k,y_r)\in \frac{d}{2}\Zd$. Then \eqref{eq:power-of-xi} may be $1$ or $-1$, so $g^a=\pm 1$. In the case $g^a=-1$ the element $g$ has order $2a$ and not $a$, however we can fix this by multiplying $g$ by $\zeta^{d/a}$.

\end{proof}

\begin{prop}
\label{prop-Heis-unique}
For two Heisenberg extensions $\tau\colon G\to M$ and $\tau'\colon G'\to M$ there exists a (not necessarily unique) isomorphism $\iota\colon G\to G'$ of groups such that $\tau'\circ\iota=\tau$.
\end{prop}

\begin{proof}
It is enough to prove the statement in the case when $\tau\colon G\to M$ is the Heisenberg extension of $(M,\phi)$ defined by \eqref{eq:Heis-1}, \eqref{eq:Heis-2}.

Let $\tau'\colon G'\to M$ be another Heisenberg extension. Let $Y'_k\in G'$ be a lift of $y_k$ in $G'$ having the same order. Then by \Cref{def:Heis}, the elements $Y'_1,\ldots,Y'_t$ must satisfy relations \eqref{eq:Heis-1}, \eqref{eq:Heis-2}. Then there is a group homomorphism $\iota\colon G\to G'$ sending $Y_k$ to $Y'_k$ and sending $\zeta\in G$ to $\zeta\in G'$. Then we automatically have $\tau'\circ \iota=\tau$. 

Let us check that $\iota$ is an isomorphism. Since $G$ and $G'$ have the same cardinalities, it is enough to check the injectivity of $\iota$. But since we have $ \tau'\circ\iota=\tau$, only elements of $\ker\tau=\mu_{\dbar}$ may be in the kernel of $\iota$. But since $\iota$ sends $\zeta$ to $\zeta$, the kernel of $\iota$ is trivial.
\end{proof}

\begin{df}
Denote by $Heis(M,\phi)$ the Heisenberg extension of $(M,\phi)$ (this group is defined up to an isomorphism). We may sometimes write simply $Heis(M)$ when the form $\phi$ is clear from the context.
\end{df}

The proof of \Cref{prop-Heis-unique} implies the following.

\begin{lem}
	\label{lem:lift-give-genrel}
	Let $(y_1,y_2,\ldots,y_t)$ be an arbitrary quasi-basis of $M$. Take lifts $Y_1,Y_2,\ldots,Y_t$ to $Heis(M,\phi)$ of these elements having the same orders. Then $Y_1,\ldots,Y_t,\zeta$ generate the group $Heis(M,\phi)$ and \eqref{eq:Heis-1}, \eqref{eq:Heis-2} is the full list of relations.	
\end{lem}

\begin{lem}
	\label{lem:Heis-Sp-surj}
	For each symplectic automorphism $\psi\in Sp(M,\phi)$ we have an automorphism $\Psi$ of $Heis(M,\phi)$ such that $\psi\circ \tau=\tau\circ \Psi$ and $\Psi(\zeta)=\zeta$. 
\end{lem}
\begin{proof}
	Take an arbitrary quasi-basis $y_1,\ldots,y_t$ of $M$. Then $\psi$ sends this quasi-basis to another quasi-basis $(y'_1,\ldots,y'_t)$ with the same orders of elements. We can lift the quasi-basis to elements $Y_1,\ldots,Y_t\in Heis(M,\phi)$ (with the same orders) and we can lift similarly the second quasi-basis to $Y'_1,\dots,Y'_t\in Heis(M,\phi)$. By \Cref{lem:lift-give-genrel} the elements $Y_1,\ldots,Y_t,\zeta$ generate $Heis(M,\phi)$ and satisfy exactly the same relations as $Y'_1,\ldots,Y'_t,\zeta$. Then we can find an automorphism of $Heis(M,\phi)$ sending $Y_k$ to $Y'_k$ for $k\in [1;t]$.
\end{proof}

Conversely, it is clear from \Cref{def:Heis} $(b)$ that each automorphism of $Heis(M,\phi)$ fixing $\zeta$ induces (via $M\simeq Heis(M,\phi) /\mu_{\dbar}$) a symplectic automorphism of $(M,\phi)$. So we get a surjective group homomorphism $\Aut_\zeta(Heis(M,\phi))\to Sp(M,\phi)$, where $\Aut_\zeta$ denotes the subgroup of the full automorphism group $\Aut$ consisting of automorphisms fixing $\zeta$.

\subsection{$\zeta$-groups} 

\begin{df}
	For a finite group $G$ with an element $\zeta$ we say that $(G,\zeta)$ is a \emph{$\zeta$-group} (or simply that $G$ is a $\zeta$-group if $\zeta$ is clear) if $\zeta$ is a central element of order $\dbar$ in $G$. In cases where there is a risk of confusion, we may sometimes write $\zeta_G$ to distinguish the element $\zeta_G\in G$ from the primitive $\dbar$-th root $\zeta_\bbC\in \bbC$.
\end{df}

By definition, a $\zeta$-group contains a subgroup isomorphic to $\mu_\dbar$ generated by $\zeta$.

\begin{df}
	For a $\zeta$-group $(G,\zeta)$ denote by $\Aut_\zeta(G)$ the subgroup of the full automorphism group $\Aut(G)$ consisting of automorphisms that fix $\zeta$.
\end{df}

\begin{df}
	We say that $V$ is a \emph{$\zeta$-representation} of a $\zeta$-group $G$ if $V$ is a representation (over $\bbC$) of $G$ such that the element $\zeta_G\in G$ acts by the scalar $\zeta_\bbC$. 
\end{df}
In other words, a $\zeta$-representation of $G$ is nothing else than a representation of the following quotient of the group algebra: $\bbC G/(\zeta_G-\zeta_\bbC)$.

\begin{df}
	For two $\zeta$-groups $G_1$ and $G_2$ we define their $\zeta$-product $G_1\times_{\zeta} G_2$ as the direct product $G_1\times G_2$ modulo the relation identifying $\zeta_{G_1}$ with $\zeta_{G_2}$.
\end{df}

It is clear that if we have a $\zeta$-representation $V_1$ of $G_1$ and a $\zeta$-representation $V_2$ of $G_2$, then the tensor product $V_1\otimes V_2$ has a structure of a $\zeta$-representation of $G_1\times_\zeta G_2$. Moreover, if these representations $V_1$ and $V_2$ are irreducible, then the representation $V_1\otimes V_2$ of $G_1\times_\zeta G_2$ is also irreducible.

\begin{df}
	We say that a $\zeta$-group $(G,\zeta)$ is \emph{simple} (as a $\zeta$-group) if it has a unique (up to isomorphism) irreducible $\zeta$-representation. 
\end{df}
If $V$ is the unique irreducible $\zeta$-representation of $(G,\zeta)$ then we have an isomorphism of algebras $\bbC G/(\zeta_G-\zeta_\bbC)\simeq \End(V)$ (since every finite-dimensional simple algebra over $\mathbb{C}$ is isomorphic to $\End(V)$ for its unique irreducible representation $V$).

\begin{lem}
	\label{lem:product-of-simples}
	A $\zeta$-product of two simple $\zeta$-groups is simple.
\end{lem}
\begin{proof}
	Let $G_1$ and $G_2$ be two simple $\zeta$-groups. It is well known that the irreducible representations of the usual direct product $G_1\times G_2$ are of the form $V_1\otimes V_2$. Moreover, such a representation $V_1\otimes V_2$ of $G_1\times G_2$ is a $\zeta$-representation of $G_1\times_\zeta G_2$ if and only if $V_1$ and $V_2$ are $\zeta$-representations of $G_1$ and $G_2$ respectively.
	
	Now, if $V_1$ is the unique $\zeta$-representation of $G_1$ and $V_2$ is the unique $\zeta$-representation of $G_2$, then 
	$V_1\otimes V_2$ is the unique $\zeta$-representation of $G_1\times_\zeta G_2$.
\end{proof}

	Assume that $(G,\zeta)$ is a simple $\zeta$-group, let $V$ be its irreducible representation. 
	\begin{lem}
		The natural homomorphism $G\to GL(V)$ is injective.
	\end{lem}
\begin{proof}
	This follows from the isomorphism of algebras $\bbC G/(\zeta_G-\zeta_\bbC)\simeq \End(V)$.
\end{proof}
Since the representation $V$ of $G$ is irreducible, there exists a unique (up to scalar) $G$-invariant (hermitian) scalar product on $V$. In particular, in view of the lemma above, we can see $G$ as a subgroup of the unitary group $U(V)$.

\begin{lem}
	\label{lem:Aut-zetagr-irr-int}
	Each automorphism of the group $G$ fixing $\zeta$ is induced by conjugation by an element of $U(V)$.
\end{lem}
\begin{proof}
	Let $\psi\colon G\to G$ be an automorphism of $G$ that fixes $\zeta$. Denote by $\rho$ the inclusion of $G$ into $GL(V)$. Set $\rho'=\rho\circ \psi$. There are two $G$-representation structures on the vector space $V$: the usual one (given by $\rho$) and the twisted one (given by $\rho'$). Both of them are irreducible $\zeta$-representations. However, since the $\zeta$-group $G$ has a unique $\zeta$-representation up to isomorphism, they must be isomorphic, i.e., there exists a linear map $f\colon V\to V$ such that $\rho'=f\circ \rho\circ f^{-1}$. Moreover, since the $G$-invariant scalar product on $V$ is unique up to scalar, we can multiply $f$ by a scalar to get an element of $U(V)$. We see now that the automorphism $\psi$ of $G$ is given by conjugation by $f$. 
\end{proof}

\subsection{Description of $Heis(M,\phi)$ for symplectic $(M,\phi)$}
\label{subs:descr-Heis-symp}

Let us describe the group $Heis(M,\phi)$ for all symplectic $(M,\phi)$. Assume first that $M$ is free and symplectic. 

\begin{lem}
\label{lem:Heis=Pauli}
	Assume that $M$ is a free finite $\Zd$-module of rank $2n$ and $\phi$ is symplectic. Then $Heis(M,\phi)$ is isomorphic to the Pauli group $\calP_n$.
\end{lem}
\begin{proof}
Take a symplectic basis $(e_1,\ldots,e_n,f_1,\ldots,f_n)$ of $M$. Let $Z_1,\ldots,Z_n,X_1,\ldots,X_n$ be the lifts to $Heis(M,\phi)$ of the elements of this basis (having the same order $d$). Then by \Cref{lem:lift-give-genrel} the group $Heis(M,\phi)$ is defined by generators $X_1,\ldots,X_n,Z_1,\ldots,Z_n,\zeta$ and the following relations:
$$
\zeta^{\dbar}=1,\qquad X_k\zeta=\zeta X_k,\qquad Z_k\zeta=\zeta Z_k,\qquad X_kX_r=X_rX_k,\qquad Z_kZ_r=Z_rZ_k,
$$
$$
X_kZ_r=Z_rX_k \mbox{ \rm(if $r\ne k$)}, \qquad X_kZ_k=\zeta^2Z_kX_k,\qquad X_k^d=1, \qquad Z_k^d=1.
$$

This is exactly the set of generators and relations for $\calP_n$, see \Cref{coro:Pn-gen+rel}.

\end{proof}

Let us now remove the assumption that $M$ is free.

\begin{lem}
	\label{lem:Heis-for-sum}
	Assume that $M_1$ and $M_2$ are two finite $\Zd$-modules with alternating forms $\phi_1$ and $\phi_2$ respectively. Consider the $\Zd$-module $M_1\oplus M_2$ together with the induced form $\phi_1\oplus \phi_2$ (making the two direct factors orthogonal). Then we have $Heis(M_1\oplus M_2,\phi_1\oplus \phi_2)\simeq Heis(M_1,\phi_1)\times_\zeta Heis(M_2,\phi_2)$.
\end{lem}
\begin{proof}
By definition, $Heis(M_1\oplus M_2,\phi_1\oplus \phi_2)$ is the unique group up to isomorphism with a homomorphism to $M_1\oplus M_2$ satisfying properties $(a)$, $(b)$, $(c)$ in \Cref{def:Heis}. On the other hand, the homomorphisms $Heis(M_1,\phi_1)\to M_1$ and $Heis(M_2,\phi_2)\to M_2$ induce a homomorphism $Heis(M_1,\phi_1)\times_\zeta Heis(M_2,\phi_2)\to M_1\oplus M_2$. We only have to check that it satisfies $(a)$, $(b)$, $(c)$. Properties $(a)$ and $(b)$ are clear. Let us check $(c)$.

Take $(m_1,m_2)\in M_1\oplus M_2$. We know that there is a lift $Y_1\in Heis(M_1,\phi_1)$ of $m_1$ and a lift $Y_2\in Heis(M_2,\phi_2)$ of $m_2$ having the same orders, let us call these orders $a_1$ and $a_2$ respectively. Then the order of $(m_1,m_2)$ in $M_1\oplus M_2$ is ${\rm lcm}(a_1,a_2)$. However, the images of the elements $Y_1$ and $Y_2$ in $Heis(M_1,\phi_1)\times_\zeta Heis(M_2,\phi_2)$ commute, so the order of their product is also ${\rm lcm}(a_1,a_2)$.
\end{proof}

We know by \Cref{lem:structure-sympl} that each symplectic $M$ is isomorphic to a direct sum of elementary symplectic modules $M\simeq S_{d_1}\oplus S_{d_2}\oplus\ldots\oplus S_{d_k}$. (We could additionally assume that $1<d_1 \mid  d_2 \mid  \cdots \mid  d_k$ if we want such a presentation to be unique. However, we do not make this assumption, since everything we claim below is true without this assumption.) Then, in view of \Cref{lem:Heis-for-sum}, it is enough to identify the Heisenberg extension for the elementary symplectic $\Zd$-modules $S_{d_0}$ with $d_0\mid d$.

Let us understand $Heis(S_{d_0})$. 
Let $\calP^{(d_0)}_n\subset GL((\bbC^{d_0})^{\otimes n})$ be the Pauli group for $n$ qudits of dimension $d_0$. (We may remove the upper index if $d_0=d$.) Denote by $\calP^{(d_0,\zeta)}_n$ the subgroup of $GL((\bbC^{d_0})^{\otimes n})$ generated by the operators coming from the action of $\calP^{(d_0)}_n$ and by multiplication by the primitive $\dbar$-th root of unity $\zeta$. (Note that the group $\calP^{(d_0)}_n$ itself contains only $\overline{d_0}$-th roots of unity.) For $d=d_0$, we clearly have $\calP^{(d,\zeta)}_n=\calP^{(d)}_n=\calP_n$.

\begin{lem}
	\label{lem:Heis-elem}
	We have $Heis(S_{d_0})\simeq \calP^{(d_0,\zeta)}_1$.
\end{lem}
\begin{proof}

By \Cref{lem:constr-Heis}, the group $Heis(S_{d_0})$ is defined by generators $(X,Z,\zeta)$ and the following relations 
$$
\zeta^{\dbar}=1, \qquad ZX=\zeta^{2\frac{d}{d_0}} XZ,\qquad X^{d_0}=Z^{d_0}=1.
$$

We recognize here the Pauli group $\calP^{(d_0)}_1$ with the addition of the extra $\dbar$-th root of unity $\zeta$.

\end{proof}

Now, combining \Cref{lem:Heis-for-sum} with \Cref{lem:Heis-elem}, we get the following.
\begin{prop}
	\label{prop:strucute-sympl-Heis}
	If we have $M\simeq \bigoplus_{r=1}^k S_{d_r}$, then 
	$$
	Heis(M,\phi)\simeq \calP_1^{(d_1,\zeta)}\times_\zeta \calP_1^{(d_2,\zeta)}\times_\zeta \ldots\times_\zeta\calP_1^{(d_k,\zeta)}.
	$$
\end{prop}

Note also that we clearly have 
$$
\calP^{(d_0,\zeta)}_n\simeq  \calP_1^{(d_0,\zeta)}\times_\zeta \calP_1^{(d_0,\zeta)}\times_\zeta \ldots\times_\zeta\calP_1^{(d_0,\zeta)} \quad\mbox{ \rm($n$ times)}.
$$

Now, let us remark that the $\zeta$-group $Heis(S_{d_0})\simeq \calP^{(d_0,\zeta)}_1$ is clearly a simple $\zeta$-group because we have $\bbC\calP^{(d_0,\zeta)}_1/(\zeta_\calP-\zeta_\bbC)\simeq \End(\bbC^{d_0})$ and then $\bbC^{d_0}$ is the unique simple $\zeta$-representation. Then by \Cref{lem:Heis-for-sum} together with \Cref{lem:product-of-simples} we get the following.

\begin{lem}
	\label{lem:Heis-symple-simple}
	If $(M,\phi)$ is symplectic, then the $\zeta$-group $Heis(M,\phi)$ is simple. If $M\simeq \bigoplus_{r=1}^k S_{d_r}$, then the unique simple $\zeta$-representation of 
	$$
	Heis(M,\phi)\simeq \calP_1^{(d_1,\zeta)}\times_\zeta \calP_1^{(d_2,\zeta)}\times_\zeta \ldots\times_\zeta\calP_1^{(d_k,\zeta)}
	$$ 
	is $\bbC^{d_1}\otimes \bbC^{d_2}\otimes \ldots \otimes\bbC^{d_k}$.
\end{lem}

\begin{rk}
	\label{rk:order-d-Chinese}
	If $d_1$ and $d_2$ are positive integers such that $d_1d_2\mid d$ and $gcd(d_1,d_2)=1$, then by the Chinese remainder theorem we get $S_{d_1d_2}\simeq S_{d_1}\oplus S_{d_2}$. Using \Cref{lem:Heis-for-sum} and \Cref{lem:Heis-elem}, we transform this to the level of the Pauli groups: 
	\begin{equation}
	\label{eq:Ch-for-Pauli}
	\calP_1^{(d_1d_2,\zeta)}\simeq\calP_1^{(d_1,\zeta)}\times_\zeta \calP_1^{(d_2,\zeta)}.
	\end{equation}
	
	Now, even if $d_1,d_2,\ldots,d_k$ in $\calP_1^{(d_1,\zeta)}\times_\zeta \calP_1^{(d_2,\zeta)}\times_\zeta \ldots\times_\zeta\calP_1^{(d_k,\zeta)}$ are arbitrary positive divisors of $d$, we can always decompose each $d_r$ into powers of primes and then regroup the factors (using \eqref{eq:Ch-for-Pauli}) to get an expression of the form $\calP_1^{(d_1,\zeta)}\times_\zeta \calP_1^{(d_2,\zeta)}\times_\zeta \ldots\times_\zeta\calP_1^{(d_k,\zeta)}$ with $d_1\mid d_2\mid \ldots\mid d_k$.
\end{rk}

\begin{rk}
	\label{rk:proof-surj-Cl-Sp}
	Assume that $M$ is free of rank $2n$ and symplectic. Then, as we have seen in \Cref{lem:Heis=Pauli}, $Heis(M,\phi)$ is isomorphic to the Pauli group $\calP_n$. Then \Cref{lem:Heis-Sp-surj} together with \Cref{lem:Aut-zetagr-irr-int} and \Cref{lem:Heis-symple-simple} give an alternative proof of \Cref{prop:surj}.
\end{rk}

\begin{lem}
	\label{lem:dim-unique-rep}
	If $(M,\phi)$ is symplectic and $V$ is the unique simple $\zeta$-representation of $Heis(M,\phi)$, then we have $\# Heis(M,\phi)=\dbar\cdot (\dim V)^2$.
\end{lem}
\begin{proof}
	This follows from the isomorphism $\bbC Heis(M,\phi)/(\zeta_{Heis(M,\phi)}-\zeta_\bbC)\simeq \End(V)$.
\end{proof}

\subsection{The description of $N(H)/H$}

\begin{lem}
	Let $\tau\colon G\to M$ be a Heisenberg extension of $(M,\phi)$. Then for an arbitrary $\Zd$-submodule $N$ (equivalently, for an arbitrary subgroup), the preimage $\tau^{-1}(N)$ is a Heisenberg extension of $N$ (with the restriction of $\phi$).
\end{lem}
\begin{proof}
	The statement follows immediately from the fact that the homomorphism from $\tau^{-1}(N)$ to $N$ obviously satisfies the axioms in \Cref{def:Heis}.
\end{proof}

\begin{thm}
	\label{thm:isom-NH/H-Heis-gen}
	Let $H$ be a commutative subgroup of $\calP_n$ satisfying \eqref{eq:cond-H-1}. Then $N(H)/H$ is isomorphic to the Heisenberg extension of $\tau(H)^\perp/\tau(H)$ (with the form induced by the symplectic form $\phi$ on $\sM$).
\end{thm}
\begin{proof}
	Since $N(H)$ is the same thing as the centralizer of $H$ in $\calP_n$, see \Cref{rk:norm=centr}, an element $g\in \calP_n$ commutes with $H$ if and only if its image $\tau(g)$ is orthogonal to $\tau(H)$. In other words, we get that $N(H)$ is the preimage of $\tau(H)^\perp$ with respect to $\tau$. Now, consider the homomorphism $N(H)/H\to \tau(H)^\perp/\tau(H)$. Let us check that it satisfies the axioms in \Cref{def:Heis}. Axiom $(b)$ is immediate. Let us check $(a)$ and $(c)$.
	
	The kernel of $N(H)/H\to \tau(H)^\perp/\tau(H)$ is isomorphic to $\tau^{-1}(\tau(H))/H$. However, $\tau^{-1}(\tau(H))$ is generated by the subgroups $H$ and $\mu_{\dbar}$ of $\calP_n$, and these two subgroups have trivial intersection by \eqref{eq:cond-H-1}. Thus the kernel of $N(H)/H\to \tau(H)^\perp/\tau(H)$ is isomorphic to $\mu_{\dbar}$, which proves $(a)$.
	
	Let us now check $(c)$. Let $m$ be an element of $\tau(H)^\perp$. 
	Let $a$, $b$ be the positive integers such that $ab$ is the order of $m\in \tau(H)^\perp$ and $a$ is the order of the image $\bar m\in \tau(H)^\perp/\tau(H)$ of $m$. Let $g\in N(H)$ be an element of $\tau^{-1}(m)$ having order $ab$. 
	
	We have $\tau(g^a)\in \tau(H)$, which implies that $g^a\in \omega H$ for some $\omega\in \mu_{\dbar}$. On the other hand, we have $g^{ab}=1$, and $g^{ab}=(g^a)^b\in \omega^b H$. In view of \eqref{eq:cond-H-1}, we must have $\omega^b=1$, i.e., $\omega$ must be a $b$-th root of unity. Then we can write $\omega=\xi^{c\frac{d}{b}}$ for some $c\in\bbZ$. On the other hand, $ab$ is a divisor of $d$. Consider $g'=\xi^{-c\frac{d}{ab}}g$. Then we have $g'^a=\xi^{-c\frac{d}{b}}g^a\in\omega^{-1}\omega H=H$. This shows that we have found a lift $\bar{g'}$ of $\bar m$ of the same order $a$.
\end{proof}

The following lemma was already proved in \cite[Theorem 1]{gheorghiu2014standard}. However, we give an alternative proof using our approach to Heisenberg groups.
\begin{lem}
	\label{lem:dim-stab-gen}
	Let $H\subset \calP_n$ be a commutative subgroup satisfying \eqref{eq:cond-H-1}. Then we have $\dim V^H=\frac{\dim V}{\# H}$.
\end{lem}
\begin{proof}
	
	Since we have an isomorphism of groups $\mu_d\simeq \Zd$, the group $H^*=\Hom(H,\Zd)$ can be identified with the group $\Hom(H,\bbC^*)=\Hom(H,\mu_d)$ of characters of $H$. Let $V=\bigoplus_{\chi\in H^*}V_\chi$ be the decomposition into common eigenspaces of $H$. Then $V^H$ is $V_\chi$ for the trivial character $\chi$.

	The group $\calP_n$ acts on the set $H^*\simeq \tau(H)^*$ by $(h\cdot \chi)(m)=\chi(m)-\phi(\tau(h),m)$, where $h\in \calP_n$, $\chi\in \tau(H)^*$, $m\in \tau(H)$. For this action we have $h(V_\chi)=V_{h\cdot \chi}$. In particular, $h$ induces a linear isomorphism between $V_\chi$ and $V_{h\cdot\chi}$.

	We see that for two characters $\chi_1$ and $\chi_2$ in the same orbit we have a linear isomorphism $V_{\chi_1}\simeq V_{\chi_2}$. So, to prove the dimension formula, it is enough to show that this action is transitive, since we then get $\dim V^H=\frac{\dim V}{\# H^*}$.
	
	Let us prove this transitivity. For this, we need to show that each character is in the orbit of the trivial character, which is represented by the zero map in $\Hom(\tau(H),\Zd)$. So, we have to show that for each $\chi\in \tau(H)^*$ we can find $h\in \calP_n$ such that $\chi=\phi(\tau(h),\bullet)$. This follows from \Cref{lem:star-surj} applied to $M=\sM$ and $N=\tau(H)$.
\end{proof}

Since by \Cref{coro:N-perp-perp} we have $(\tau(H)^{\perp})^\perp=\tau(H)$, the $\Zd$-module $\tau(H)^{\perp}/\tau(H)$ with the induced alternating form is symplectic. Then by \Cref{prop:strucute-sympl-Heis}, the group $Heis(\tau(H)^{\perp}/\tau(H))$ is isomorphic to a group of the form  
$$
\calP_1^{(d_1,\zeta)}\times_\zeta \calP_1^{(d_2,\zeta)}\times_\zeta \ldots\times_\zeta\calP_1^{(d_k,\zeta)},
$$
see also \Cref{rk:order-d-Chinese}.
\begin{thm}
	\label{thm:main-gen+rep}
	There is a unitary isomorphism 
	$$
	V^H\simeq \bbC^{d_1}\otimes \bbC^{d_2}\otimes \ldots \otimes\bbC^{d_k}
	$$
	such that the isomorphism of groups 
	$$
	N(H)/H\simeq Heis(\tau(H)^\perp/\tau(H))\simeq \calP_1^{(d_1,\zeta)}\times_\zeta \calP_1^{(d_2,\zeta)}\times_\zeta \ldots\times_\zeta\calP_1^{(d_k,\zeta)}
	$$ 
	in \Cref{thm:isom-NH/H-Heis-gen} intertwines the action of $N(H)/H$ on $V^H$ with the action of 
	$$
	\calP_1^{(d_1,\zeta)}\times_\zeta \calP_1^{(d_2,\zeta)}\times_\zeta \ldots\times_\zeta\calP_1^{(d_k,\zeta)} \quad\mbox{on}\quad \bbC^{d_1}\otimes \bbC^{d_2}\otimes \ldots \otimes\bbC^{d_k}.
	$$
\end{thm}

\begin{proof}
	By \Cref{lem:Heis-symple-simple}, this group is simple as a $\zeta$-group and its unique irreducible $\zeta$-representation is $V_0=\bbC^{d_1}\otimes \bbC^{d_2}\otimes \ldots \otimes \bbC^{d_k}$. 
Then by \Cref{lem:dim-unique-rep}, we have 
	$$
	\dim V_0=\sqrt{\frac{\#Heis(\tau(H)^{\perp}/\tau(H))}{\dbar}}.
	$$
	
	On the other hand, by \Cref{lem:dim-stab-gen}, the dimension of the representation $V^H$ of $N(H)/H$ satisfies  
	$$
	\dim(V^H)=\frac{\dim V}{\#H}.
	$$

	To identify $V^H$ with $V_0=\bbC^{d_1}\otimes \bbC^{d_2}\otimes \ldots \otimes \bbC^{d_k}$ we only have to prove that they have the same dimensions.
	We have $\#(\tau(H)^\perp)=\frac{\#\sM}{\#H}$ and then 
	$$
	\#(\tau(H)^\perp/\tau(H))=\frac{\#\sM}{(\#H)^2}=\frac{d^{2n}}{(\#H)^2}.
	$$ 
	Then we get 
	$$
	\dim V_0=\sqrt{\frac{\#Heis(\tau(H)^{\perp}/\tau(H))}{\dbar}}=	\sqrt{\#(\tau(H)^\perp/\tau(H))}=\sqrt{\frac{d^{2n}}{(\#H)^2}}=\frac{d^n}{\# H}=\dim V^H.
	$$
	
	The isomorphism of representations $V^H\simeq V_0$ can be chosen unitary by the uniqueness up to scalar of the invariant scalar product.
\end{proof}

\subsection{Shifted free subgroups}
\label{subs:shift}

We have seen in \Cref{thm:main-free-dstab} that if an abelian subgroup $H\subset \calP_n$ satisfies \eqref{eq:cond-H-1} and \eqref{eq:cond-H-2}, then the action of $N(H)/H$ on $V^H$ may be identified with the action of $\calP_{n-k}$ on $(\bbC^d)^{\otimes (n-k)}$. We see in \Cref{thm:isom-NH/H-Heis-gen} and \Cref{thm:main-gen+rep} that if we remove the freeness assumption \eqref{eq:cond-H-2}, then some of the qudits may be reduced to qudits of smaller dimension. It would be useful to identify the optimal replacement of the freeness condition \eqref{eq:cond-H-2} such that the conclusion of \Cref{thm:main-free-dstab} is still true (i.e., such that no such reduction happens).

Let $H'\subset \calP_n$ be an abelian subgroup satisfying \eqref{eq:cond-H-1} and \eqref{eq:cond-H-2}. Then we can find a symplectic basis $(e_1,\ldots,e_n,f_1,\ldots,f_n)$ of $\sM$ such that $(e_1,\ldots,e_k)$ is a $\Zd$-basis of $\tau(H')$, where $k$ is such that $\#H'=d^k$. By \Cref{lem:lift-give-genrel}, this basis can be lifted to elements $Z'_1,\ldots,Z'_n, X'_1,\ldots,X'_n$ satisfying exactly the relations of the standard generators of the Pauli group in \Cref{coro:Pn-gen+rel} (in fact, we can send them to the standard generators by an automorphism, see the proof of \Cref{lem:Heis-Sp-surj}) and such that $H'=\langle Z'_1,Z'_2,\ldots, Z'_k \rangle$.

\begin{df}
	We say that a group $H$ is a \emph{shift} of $H'$ if there are some positive integers $(a_1,a_2,\ldots,a_k)$ dividing $d$ such that we have
	$$
	H=\langle {Z'}_1^{a_1},{Z'}_2^{a_2},\ldots, {Z'}_k^{a_k},{X'_1}^{d/a_1},\ldots,{X'_k}^{d/a_k} \rangle
	$$
	(for some choice of the symplectic basis $(e_1,\ldots,e_n,f_1,\ldots,f_n)$ and lifts $Z'_1,\ldots,Z'_n, X'_1,\ldots,X'_n$ as above).
	Let us call the groups $H$ of this form \emph{shifted free}.
\end{df}

\begin{rk}
	\label{rk:def-shifted-free}
	In view of \Cref{prop:class-Lagr} and \Cref{prop:classif-Lang}, the condition of being "shifted free" is nothing else than "$\tau(H)$ is a Lagrangian inside a free symplectic $\Zd$-submodule of $\sM$".
\end{rk}

\begin{thm}
	\label{thm:shifted-free}
	Let $H$ be an abelian subgroup of $\calP_n$ satisfying \eqref{eq:cond-H-1}. Then the group $N(H)/H$ is isomorphic to $\calP_{n-k}$ for some $k$ if and only if $H$ is shifted free. Moreover, in this case the action of $N(H)/H$ on $V^H$ can be identified with the action of $\calP_{n-k}$ on $(\bbC^d)^{\otimes (n-k)}$, where the isomorphism $V^H\simeq (\bbC^d)^{\otimes (n-k)}$ is unitary.
\end{thm}

\begin{proof}
	Let us first show that we have $N(H)/H\simeq \calP_{n-k}$ if $H$ is shifted free. We have
	$$
    H=\langle {Z'}_1^{a_1},{Z'}_2^{a_2},\ldots, {Z'}_k^{a_k},{X'_1}^{d/a_1},\ldots,{X'_k}^{d/a_k} \rangle
    $$
	Then
	$$
	\tau(H)=\langle a_1e_1,\ldots,a_ke_k,\frac{d}{a_1}f_1,\ldots,\frac{d}{a_k}f_k \rangle
	$$
	and therefore
	$$
	\tau(H)^\perp=\langle a_1e_1,\ldots,a_ke_k,e_{k+1},\ldots,e_n,\frac{d}{a_1}f_1,\ldots,\frac{d}{a_k}f_k,f_{k+1},\ldots,f_n \rangle .
	$$
	We see that $\tau(H)^\perp/\tau(H)$ is a free symplectic $\Zd$-module of rank $2(n-k)$. Then, by \Cref{thm:isom-NH/H-Heis-gen} and \Cref{prop:strucute-sympl-Heis}, we have $N(H)/H\simeq \calP_{n-k}$.
	
	Now assume that $N(H)/H\simeq \calP_{n-k}$ and let us show that $H$ is shifted free. This isomorphism of groups implies that $\tau(H)^\perp/\tau(H)$ is a free $\Zd$-module of rank $2(n-k)$. Let $e_{k+1},\ldots, e_n,f_{k+1},\ldots,f_n\in \tau(H)^\perp$ be lifts of elements of a symplectic basis of $\tau(H)^\perp/\tau(H)$. Then $N:=\langle e_{k+1},\ldots, e_n,f_{k+1},\ldots,f_n\rangle$ is a free symplectic submodule of $\sM$ of rank $2(n-k)$. Then, by \Cref{rk:orth-sympl}, $N^\perp$ is also free symplectic of rank $2k$.

	Since we have $(\tau(H)^{\perp})^\perp=\tau(H)$, then $\tau(H)$ is the maximal submodule of $\tau(H)^\perp$ orthogonal to all elements of $\tau(H)^\perp$. Now, since we have $\tau(H)^\perp=\tau(H)\oplus N$, then $\tau(H)$ is a maximal isotropic submodule of $N^\perp$. Hence $\tau(H)$ is Lagrangian in $N^\perp$. Then $H$ is shifted free, see \Cref{rk:def-shifted-free}.

	The identification of representations follows automatically from the uniqueness of the irreducible $\zeta$-representation.
\end{proof}

\subsection{Twisted free subgroups}
\label{subs:twist}

Let us introduce a notion of a \emph{twist}, which generalises the notion of a shift. 
Roughly speaking, a twist is a shift with a defect: when we shift, we lose a part of what we shift.

\begin{df}
	Let $H'$ be as in \Cref{subs:shift}. We say that a subgroup $H\subset \calP_n$ is \emph{twisted free} if there exist positive integers
	$
	a_1,\ldots,a_k,b_1,\ldots,b_k
	$
	such that $d\mid a_r b_r$ for each $r\in[1;k]$, and such that
	$$
	H=\langle {Z'}_1^{a_1},{Z'}_2^{a_2},\ldots, {Z'}_k^{a_k},{X'_1}^{b_1},\ldots,{X'_k}^{b_k} \rangle
	$$
	for some choice of a symplectic basis
	$
	(e_1,\ldots,e_n,f_1,\ldots,f_n)
	$
	of $\sM$ and some choice of lifts
	$
	Z'_1,\ldots,Z'_n,\; X'_1,\ldots,X'_n
	$
	of this basis to $\calP_n$.
\end{df}

\begin{rk}
	\label{rk:def-twisted-free}
	Equivalently (in view of \Cref{lem:isotr-in-Sd}), a subgroup $H$ is twisted free if there exists an orthogonal decomposition
	$
	\sM \simeq S_d^{\oplus n}
	$
	into $n$ copies of the elementary free symplectic $\Zd$-module $S_d$, and for each $r\in[1;k]$ an isotropic submodule
	$
	N_r \subset S_d
	$
	such that
	\[
	\tau(H)=N_1\oplus N_2\oplus \cdots \oplus N_k.
	\]
	In the special case $a_r b_r=d$ for all $r$, each $N_r$ is Lagrangian in the corresponding copy of $S_d$, and we recover the notion of a shift.
\end{rk}

For each $r\in[1;k]$, set
\[
c_r=\frac{a_r b_r}{d}.
\]
In the shifted case one has $a_r b_r=d$, hence $c_r=1$. In the twisted case, however, the integers $c_r$ measure the defect.
More precisely, for the $r$-th block one has
$
N_r^\perp/N_r \simeq S_{c_r},
$
and therefore
$$
\tau(H)^\perp/\tau(H)\simeq
S_{c_1}\oplus S_{c_2}\oplus \cdots \oplus S_{c_k}\oplus S_d^{\oplus (n-k)}.
$$
Hence, we get
\[
N(H)/H \simeq
\calP_1^{(c_1,\zeta)}
\times_\zeta
\calP_1^{(c_2,\zeta)}
\times_\zeta
\cdots
\times_\zeta
\calP_1^{(c_k,\zeta)}
\times_\zeta
\calP_{n-k},
\]
and 
\[
V^H \simeq
\C^{c_1}\otimes \C^{c_2}\otimes \cdots \otimes \C^{c_k}\otimes (\C^d)^{\otimes (n-k)}.
\]

Thus, compared to the shifted case, twisted free subgroups produce additional
smaller qudits of dimensions $c_1,\ldots,c_k$ in the protected space.
When all $c_r=1$, these extra factors disappear and we recover the usual shifted free picture.

\begin{ex}
	Consider one qudit of dimension $8$ ($n=1, d=8$). Let $H$ be the subgroup of $\calP^{(8)}_1$ generated by $X^4$ and $Z^4$. This group is a twist of the subgroup generated by $Z$. Then we have $V^H\simeq \bbC^2$, and $N(H)/H$ is isomorphic to the one-qubit Pauli group with an additional $16$th root of unity. This example shows how one qubit can be encoded in a qudit of dimension $8$.
\end{ex}

\section{Twisted Kitaev model}
\label{sec:Kitaev-tw-sh}

In the Kitaev model (\Cref{subs:Kitaev-mod}), the subgroup $H\subset \calP_{\#E}$ was constructed as the group generated by the operators of the form $A_s$, $s\in S$, and $B_f$, $f\in F$ (but we can remove one $A_{s_0}$ and one $B_{f_0}$ from the set of generators). On the other hand, the notions of shifts and twists were introduced in \Cref{subs:shift} and \Cref{subs:twist}. We would like to look at some examples of what these notions give for the Kitaev model.

\subsection{Full shift}

	Fix $s_0,s_1\in S$ and fix a path $t$ from $s_0$ to $s_1$. Consider the following modification of the group $H$. 
	Let us denote by $H$ the group generated by $A_s$, $s\in S\backslash\{s_0,s_1\}$, by $B_f$, $f\in F$, and (to compensate for the removed $A_{s_1}$ and $A_{s_0}$) by the path operator $S^Z(t)$.  
	
	The combinatorial description of the new characters $\chi\in H^*$ is as follows. We still have notions of electric charges on all vertices (except for $s_0$ and $s_1$), and we still have the notion of magnetic charges at faces (with the total charge equal to zero modulo $d$). Additionally, we now have a notion of the current on $t\colon s_0\to s_1$, which takes values in $\Zd$. In other words, we have created the notion of a current from $s_0$ to $s_1$, but because of this we lost the notions of charges at $s_0$ and $s_1$ themselves.
	
	We could naturally generalise this construction by removing $A_{s_0}, A_{s_1}, \ldots, A_{s_u}$ and by adding path operators $S^Z(t_1), S^Z(t_2), \ldots, S^Z(t_u)$ for some $t_1\colon s_0\to s_1$, $t_2\colon s_0\to s_2$, $\ldots$, $t_u\colon s_0\to s_u$.
	Then we lose the notions of charges at $s_0,s_1,\ldots,s_u$, but we obtain the notions of currents between these chosen vertices.
	
	Since we are in the "shifted free" situation, the new protected space $V^H$ still has the form $(\bbC^d)^{\otimes {2g}}$.

\subsection{Partial shift}

	We can upgrade the example above. Fix a decomposition $d=ab$ with positive integers $a$ and $b$, and fix vertices $s_0,s_1$, and a path $t\colon s_0\to s_1$. Then we consider the following $H$. We remove $A_{s_0}$ and $A_{s_1}$, but we add $A_{s_1}^a$ (we can also add $A_{s_0}^a$, this changes nothing because of $\prod_{s\in S}A_s=I$), and we add the power $S^Z(t)^b$ of the path operator. Now, in the combinatorial description of the eigenspaces, we have the current between $s_0$ and $s_1$ (taking values in $b\bbZ/d\bbZ\subset \bbZ/d\bbZ=\Zd$), and the old charges at $s_0$ and $s_1$ modulo $d$ are not defined anymore, but there is still a notion of charges at $s_0$ and $s_1$ modulo $b$. This looks natural in the sense that since the current is divisible by $b$, the charges are still defined at least modulo $b$. 
	
	We could naturally generalise this construction. Fix $s_0,s_1,\ldots,s_u$ and $t_1,t_2,\ldots,t_u$ as in the example above, and fix $d=a_1b_1=a_2b_2=\ldots=a_ub_u$. Let us remove $A_{s_0}, A_{s_1}, \ldots, A_{s_u}$, but add $A_{s_1}^{a_1}, \ldots, A_{s_u}^{a_u}$ and path operators $S^Z(t_1)^{b_1}, S^Z(t_2)^{b_2}, \ldots, S^Z(t_u)^{b_u}$ for some $t_1\colon s_0\to s_1$, $t_2\colon s_0\to s_2$, $\ldots$, $t_u\colon s_0\to s_u$.
	Then we get a notion, for $r\in[1;u]$, of a current between $s_0$ and $s_r$ taking values in $b_r\bbZ/d\bbZ$, and the charge at $s_r$ is then defined only modulo $b_r$.
	
	Since we are in the "shifted free" situation, the new protected space $V^H$ still has the form $(\bbC^d)^{\otimes {2g}}$.

\subsection{Twisted case}

	We can now consider a situation similar to the previous example, but we do not assume $d=a_rb_r$ anymore, only $d\mid a_rb_r$. Set $c_r=\frac{a_rb_r}{d}$. Now, the current from $s_0$ to $s_r$ takes values in $b_r\bbZ/d\bbZ$, and the charge at $s_r$ is defined modulo $\frac{b_r}{c_r}=\frac{d}{a_r}$. Since we are now in the twisted case and not in the shifted case, the new protected space is $V^H\simeq (\bbC^d)^{\otimes 2g}\otimes \bbC^{c_1}\otimes\ldots\otimes\bbC^{c_u}$.

\bibliographystyle{abbrvurl}
\bibliography{dstab}

@article{hostens2005stabilizer,
  title={Stabilizer states and {C}lifford operations for systems of arbitrary dimensions and modular arithmetic},
  author={Hostens, Erik and Dehaene, Jeroen and De Moor, Bart},
  journal={Physical Review A-Atomic, Molecular, and Optical Physics},
  volume={71},
  number={4},
  pages={042315},
  year={2005},
  publisher={APS}
}

@article{albouy2008matrix,
  title={Matrix reduction and {L}agrangian submodules},
  author={Albouy, Olivier},
  journal={arXiv preprint arXiv:0809.1059},
  year={2008}
}

@article{gheorghiu2014standard,
  title={Standard form of qudit stabilizer groups},
  author={Gheorghiu, Vlad},
  journal={Physics Letters A},
  volume={378},
  number={5-6},
  pages={505--509},
  year={2014},
  publisher={Elsevier}
}

@article{freedman2001projective,
  title={Projective plane and planar quantum codes},
  author={Freedman, Michael H and Meyer, David A},
  journal={Foundations of Computational Mathematics},
  volume={1},
  number={3},
  pages={325--332},
  year={2001},
  publisher={Springer}
}

@book{gottesman1997stabilizer,
  title={Stabilizer codes and quantum error correction},
  author={Gottesman, Daniel},
  year={1997},
  publisher={California Institute of Technology}
}

@article{kitaev2003fault,
  title={Fault-tolerant quantum computation by anyons},
  author={Kitaev, A Yu},
  journal={Annals of physics},
  volume={303},
  number={1},
  pages={2--30},
  year={2003},
  publisher={Elsevier}
}

@article{maksimau2025spin,
  title={What is spin? {I}ntroduction to qubits},
  author={Maksimau, Ruslan},
  year={2025},
url={https://hal.science/hal-05557990v1/document}
}

\address{R. M.: Laboratoire Analyse, G\'eom\'etrie et Mod\'elisation, CY Cergy Paris Universit\'e, 95302  Cergy-Pontoise (France),}
\email{ruslmax@gmail.com, ruslan.maksimau@cyu.fr}

\end{document}